\documentclass{birkart}
\usepackage{amsmath}
\usepackage{amssymb}
\usepackage{amsfonts}
\usepackage{graphicx}
\usepackage{texdraw}
\usepackage{graphpap}

\input txdtools

 \newtheorem{thm}{Theorem}[section]
 \newtheorem{thml}{Theorem}
 \newtheorem{thmlb}{Theorem}
 \newtheorem{cor}[thm]{Corollary}
 \newtheorem{lem}[thm]{Lemma}
 \newtheorem{prop}[thm]{Proposition}
 \theoremstyle{definition}
 \newtheorem{defn}[thm]{Definition}

 \theoremstyle{remark}
 
 \newtheorem{rem}[thm]{Remark}

\numberwithin{equation}{section}
\numberwithin{figure}{section}

\newcommand{\proj}{{\mathbf Q}}

\newcommand{\e}{\mathrm e}

\newcommand{\C}{{\mathbb C}}
\newcommand{\D}{{\mathbb D}}

\newcommand{\T}{{\mathbb T}}
\newcommand{\R}{{\mathbb R}}

\newcommand{\Z}{{\mathbb Z}}

\newcommand{\setS}{{\Sigma}}

\newcommand{\Uset}{\mathcal{U}}

\newcommand{\Vset}{\mathcal{V}}

\newcommand{\Iop}{\mathbf I}
\newcommand{\Pop}{\mathbf P}
\newcommand{\Rop}{\mathbf R}
\newcommand{\Eop}{\mathbf E}
\newcommand\BV{\mathrm{BV}}

\newcommand{\tpar}{v}
\newcommand{\sigmaalg}{\mathfrak{S}}
\newcommand{\borelalg}{\mathfrak{B}}
\newcommand{\sigmat}{\theta}
\newcommand{\Spec}{\sigma}

\newcommand{\im}{\operatorname{Im}}

\newcommand{\Tope}{{\mathbf T}}
\newcommand{\Sop}{{\mathbf S}}
\newcommand{\Qop}{{\mathbf Q}}
\newcommand{\Cop}{{\mathbf C}}

\newcommand{\Zop}{{\mathbf Z}}

\newcommand{\diff}{{\mathrm d}}
\newcommand{\imag}{{\mathrm i}}

\begin{document}
%
\title[Perron-Frobenius operators and the Klein-Gordon equation]
{Perron-Frobenius operators and the Klein-Gordon \\ equation}


\author[Canto]
{Francisco Canto-Mart\'\i{}n}

\address{Canto-Mart\'\i{}n: Department of Mathematical Analysis\\
University of Sevilla\\
Sevilla\\
SPAIN}

\email{fcanto@us.es}

\author[Hedenmalm]
{H\aa{}kan Hedenmalm}

\address{Hedenmalm: Department of Mathematics\\
The Royal Institute of Technology\\
S -- 100 44 Stockholm\\
SWEDEN}

\email{haakanh@math.kth.se}

\thanks{Research partially supported by Plan Nacional ref MTM2009-09501.
Research of the first author partially supported by Ministerio de Educaci\'on
(FPU). Research of the second author partially supported by the G\"oran 
Gustafsson Foundation (KVA) and by Vetenskapsr\aa{}det (VR).
Research of the third author partially supported by Junta de Andaluc\'\i{}a
ref FQM260.}


\author[Montes]
{Alfonso Montes-Rodr\'\i{}guez}

\address{Montes-Rodr\'\i{}guez: Department of Mathematical Analysis\\
University of Sevilla\\
Sevilla\\
SPAIN}

\email{amontes@us.es}





\begin{abstract} 
For a smooth curve $\Gamma$ and a set $\Lambda$ in the plane $\R^2$, let 
$\mathrm{AC}(\Gamma;\Lambda)$ be the space of finite Borel
measures in the plane supported on $\Gamma$, absolutely continuous
with respect to the arc length and whose Fourier transform vanishes
on $\Lambda$. Following \cite{hh}, we say that $(\Gamma,\Lambda)$ is a 
Heisenberg uniqueness pair if $\mathrm{AC}(\Gamma;\Lambda)=\{0\}$.  
In the context of a hyperbola $\Gamma$, the study of Heisenberg uniqueness 
pairs is the same as looking for uniqueness sets $\Lambda$ of a collection 
of solutions to the Klein-Gordon equation.
In this work, we mainly address the issue of finding the dimension of 
$\mathrm{AC}(\Gamma;\Lambda)$ when it is nonzero. We will fix the curve 
$\Gamma$ to be the hyperbola $x_1x_2=1$, and the set $\Lambda=
\Lambda_{\alpha,\beta}$ to be the lattice-cross
\[
\Lambda_{\alpha,\beta}=\left(\alpha \Z\times\{0\}\right)\cup 
\left(\{0\}\times\beta \Z\right),
\] 
where $\alpha,\beta$ are positive reals. We will also consider $\Gamma_+$,
the branch of $x_1x_2=1$ where $x_1>0$. In \cite{hh}, it is shown that 
$\mathrm{AC}(\Gamma;\Lambda_{\alpha,\beta})=\{0\}$ if and only if 
$\alpha\beta\le1$. 
Here, we show that for $\alpha\beta>1$, we get a rather drastic ``phase 
transition'': $\mathrm{AC}(\Gamma;\Lambda_{\alpha,\beta})$ is 
infinite-dimensional whenever $\alpha\beta>1$.  It is shown in \cite{HM2} 
that $\mathrm{AC}(\Gamma_+;\Lambda_{\alpha,\beta})=\{0\}$ if and only if 
$\alpha\beta<4$. Moreover, at the edge $\alpha\beta=4$, the behavior is more 
exotic: the space $\mathrm{AC}(\Gamma_+;\Lambda_{\alpha,\beta})$ is 
one-dimensional. Here, we show that the dimension of 
$\mathrm{AC}(\Gamma_+;\Lambda_{\alpha,\beta})$ is infinite 
whenever $\alpha\beta>4$. Dynamical systems, and more specifically 
Perron-Frobenius operators, will play a prominent role in the presentation.
\end{abstract}

\maketitle

\addtolength{\textheight}{2.2cm}







\begin{section}{Introduction}

\begin{subsection}{Background: the Heisenberg uncertainty principle}
The Heisenberg uncertainty principle asserts that it is not possible to have 
completely accurate information about the position and the momentum of a 
particle at the same time. If $\psi$ is the spatial wave-function, which 
describes the position of the particle in question, and it is known that 
$\psi$ is concentrated to a small region,  then the deviation of the momentum
wave-function of $\psi$ from its mean must be large. The momentum 
wave-function is essentially the Fourier transform of the spatial 
wave-function.
So, we may consider the Heisenberg uncertainty principle as the mathematical 
statement that a function and its Fourier transform cannot both be too 
concentrated simultaneously; cf. \cite{ben}, \cite{hard}, and \cite{hav}.
\end{subsection}

\begin{subsection}{Heisenberg uniqueness pairs} 
Let $\Gamma$ be a finite disjoint union of smooth curves
in the plane and $\Lambda$ a subset of the plane. Let
$\mathrm{AC}(\Gamma;\Lambda)$ be the space of bounded Borel measures $\mu$ in
the plane supported on $\Gamma$, absolutely continuous with respect
to arc length and whose Fourier transform
\begin{equation}
\widehat{\mu}(x_1,x_2)=\int_{\Gamma}
\e^{\imag\pi(x_1y_1+x_2y_2)}\diff\mu(y_1,y_2),
\qquad (x_1,x_2)\in\R^2,
\label{eq-FT}
\end{equation}
vanishes on $\Lambda$.  Following \cite{hh}, we say that
$(\Gamma,\Lambda)$ is a {\em Heisenberg uniqueness pair} if 
$\mathrm{AC}(\Gamma;\Lambda)=\{0\}$.
\\
\noindent
When $\Gamma$ is an algebraic curve, that is, the zero locus
of a polynomial $p$ in two variables with real coefficients, the
requirement that the support of $\mu$ be contained in $\Gamma$
means that $\widehat{\mu}$ solves the partial
differential equation
\begin{equation}\label{ecdp}
p\left(\frac{\partial_ {x_1}}{\pi i},\frac{\partial_{x_2}}{\pi
i}\right)\widehat{\mu}=0.
\end{equation}
So, there is a natural interplay between the Heisenberg uniqueness pairs and 
the theory of partial differential equations (PDE), cf. \cite{hh}. 
The most natural examples appear we consider quadratic polynomials $p$ 
corresponding to the standard conic sections: the line, two parallel
lines, two crossing lines, the hyperbola, the ellipse, and the
parabola.
The natural invariance of Heisenberg uniqueness pairs under affine
transformations of the plane allows us to reduce to the canonical
models for these curves (cf. \cite{hh}). 
The case when $\Gamma$ is either one line or the union of two 
parallel lines was solved completely for general $\Lambda\subset\R^2$ in 
\cite{hh}. In this direction, Blasi-Babot has solved particular cases when
$\Gamma$ is the union of three parallel lines, see \cite{blas}. 
The case when $\Gamma$ is a circle (which also covers the ellipse case
after an affine mapping) was recently studied independently by Lev and by 
Sj\"olin in \cite{ps}, \cite{Lev}, where e.g., circles and unions of straight 
lines are considered as sets $\Lambda$. Also subsets of the unions of
straight lines were considered, and a connection with the Beurling-Malliavin 
theory was made. Very little seems to be known when $\Gamma$ is a parabola 
or two intersecting lines.
\end{subsection}

\begin{subsection}{Heisenberg uniqueness pairs for the hyperbola} 
The case of the hyperbola 
$\Gamma:\,x_1x_2=1$ and the lattice-cross
\[
\Lambda_{\alpha,\beta}=
\left(\alpha\Z\times\{0\}\right)\cup\left(\{0\}\times\beta\Z
\right), 
\]
for given positive reals $\alpha,\beta$ was considered in \cite{hh}, 
where we the following result was obtained.

\begin{thml}[\textrm{Hedenmalm, Montes-Rodr\'{\i}guez}]\label{thh} Let 
$\Gamma$ be the hyperbola $x_1x_2=1$.
 Then $(\Gamma,\Lambda)$ is a Heisenberg uniqueness pair if and only if 
$\alpha\beta\leq 1$.
\end{thml}

When one of the branches of the hyperbola is considered, the critical density
changes, see \cite{HM2}.

\begin{thml}
[\textrm{Hedenmalm, Montes-Rodr\'{\i}guez}]\label{theobranch}  
Let $\Gamma_+$ be the branch of the hyperbola $x_1x_2=1$ where $x_1>0$.
Then $\mathrm{AC}(\Gamma_+;\Lambda_{\alpha,\beta})=\{0\}$ if and only if 
$\alpha\beta<4$. Moreover, when $\alpha\beta=4$, 
$\mathrm{AC}(\Gamma_+;\Lambda_{\alpha,\beta})$ is one-dimensional.
\end{thml}

For sub-critical density of the lattice-cross, we have the following two 
theorems, corresponding to the hyperbola and a branch of the hyperbola.

\begin{thm}\label{theop} Let $\Gamma$ be the hyperbola $x_1x_2=1$. Then 
$\mathrm{AC}(\Gamma;\Lambda_{\alpha,\beta})$ is infinite-dimensional for
$\alpha\beta>1$.
\end{thm}

\begin{thm}\label{theop100} Let $\Gamma_+$ be the branch of the hyperbola 
$x_1x_2=1$ with $x_1>0$. Then the space 
$\mathrm{AC}(\Gamma_+;\Lambda_{\alpha,\beta})$ is infinite-dimensional
for $\alpha\beta>4$.
\end{thm}

\noindent Although the proofs of Theorem \ref{theop} and
\ref{theop100} share a certain degree of parallelism, the proof of Theorem
\ref{theop} is more delicate than that of Theorem
\ref{theop100}. Mainly, the difference is that at the edge $\alpha\beta=4$,
the Perron-Frobenius operator which appears in the context of of Theorem 
\ref{theop100}, induced by the classical Gauss map, has a spectral gap 
[acting on the space of functions of bounded variation], while for the
Perron-Frobenius operator associated to the edge case $\alpha\beta=1$ 
in the context of Theorem \ref{theop} does not have such a spectral gap;
this is so because the Gauss-type transformation which defines the 
Perron-Frobenius operator has an indifferent fixed point.
\medskip

\noindent The basic invariance properties of Heisenberg uniqueness pairs allow
us to take $\alpha=1$ and we may appeal to duality and reformulate Theorem 
\ref{thh} as follows (cf. \cite{hh}).

\begin{thmlb}\label{thhh}
Let $\mathcal{M}_\beta$ be the linear subspace of $L^\infty(\R)$ spanned
by the functions $x\mapsto\e^{\imag m\pi x}$ and 
$x\mapsto\e^{\imag n\pi\beta/x}$, where $m,n$ range over the integers and 
$\beta$ is a fixed positive real. 
Then $\mathcal{M}_\beta$ is weak-star dense in $L^\infty(\R)$ if and only if 
$\beta\leq1$.
\end{thmlb}

\noindent The analogous reformulation of Theorem \ref{theop} runs as follows:

\begin{thm}\label{tpo}
Let $\mathcal{M}_\beta$ be the linear subspace of $L^\infty(\R)$ spanned
by the functions $x\mapsto\e^{\imag m\pi x}$ and 
$x\mapsto\e^{\imag n\pi\beta/x}$, 
where $m,n$ range over the integers. Then the weak-star closure of
$\mathcal{M}_\beta$ in $L^\infty(\R)$ has infinite codimension in 
$L^\infty(\R)$ for $\beta>1$.
\end{thm}

If we instead take $\alpha=2$, we may reformulate Theorems \ref{theobranch} 
and \ref{theop100} as follows. 

\begin{thmlb}
\label{tpo100} 
Let $\mathcal{N}_\beta$ be the linear subspace of
$L^\infty(\R_+)$ spanned by the functions $x\mapsto\e^{\imag 2m\pi x}$ and 
$x\mapsto\e^{\imag n\pi\beta/x}$, where $m,n$ range over the integers and 
$\beta$ is a fixed positive real. Then $\mathcal{N}_\beta$ is weak-star dense 
in $L^\infty(\R_+)$ if and only if $\beta<2$. Moreover, the weak-star
closure of $\mathcal{N}_\beta$ in $L^\infty(\R_+)$ has codimension $1$ in
$L^\infty(\R_+)$ for $\beta=2$. 
\end{thmlb}

\begin{thm}\label{tpo1001}
Let $\mathcal{N}_\beta$ be the linear subspace of $L^\infty(\R_+)$
spanned by the functions $x\mapsto\e^{\imag 2m\pi x}$ and
$x\mapsto\e^{\imag n\pi\beta/x}$, where $m,n$ range over the integers and 
$\beta$ is a fixed positive real. Then the weak-star closure of
$\mathcal{N}_\beta$ has infinite codimension in $L^\infty(\R_+)$ for $\beta>2$.
\end{thm}

\noindent By general Functional Analysis, the codimension of the weak-star 
closure of ${\mathcal{M}_\beta}$ equals the dimension of its pre-annihilator 
space
\begin{equation}
\label{aniquilator}
\mathcal{M}_\beta^\bot=\Big\{f\in L^1(\R):\int_\R
f(x)\e^{\imag n\pi x}\,\diff x=\int_\R f(x)\e^{\imag n\pi\beta/x}\,\diff x=0
\,\,\,\text{for all}\,\,\, n\in\Z\Big\}.
\end{equation}
Likewise, the codimension of the weak-star closure of ${\mathcal{N}_\beta}$ 
equals the dimension of its pre-annihilator space 
\begin{equation}
\label{aniquilatorN}
\mathcal{N}_\beta^\bot=\Big\{f\in L^1(\R_+):\int_{\R_+}
f(x)\e^{\imag 2n\pi x}\,\diff x=\int_{\R_+} f(x)\e^{\imag n\pi\beta/x}\,
\diff x=0\,\,\,\text{for all}\,\,\, n\in\Z\Big\}.
\end{equation}

\noindent If $f\in  \mathcal{N}_\beta^\bot$, then it is easy to see
that the function $g(x)=f(\frac12x)$, extended to vanish along the negative 
semi-axis $\R_-$, belongs to $\mathcal{M}_{2\beta}^\bot$. So, Theorems 
\ref{tpo} and \ref{tpo1001} show that there are elements in
$\mathcal{M}_\beta^\bot$ with support on the positive semi-axis precisely 
when $\beta\geq 4$.

\begin{cor} In the pre-annihilator $\mathcal{M}_{\beta}^\bot$ there exists 
a non-trivial element that vanishes on $\R_-$ if and only if $\beta\ge4$.
Moreover, if $\beta=4$, there is only a one-dimensional subspace of
such elements, while if $\beta>4$, there is an infinite-dimensional subspace
with this property.
\end{cor}

\begin{rem}
In the context of Theorems \ref{tpo} and \ref{tpo1001}, we actually construct
rather concrete infinite-dimensional subspaces of $\mathcal{M}_\beta^\bot$
and $\mathcal{N}_\beta^\bot$, respectively; cf. Theorems \ref{thm-8.2} and
\ref{thm-8.6}.
\end{rem}

\end{subsection}

\begin{subsection}{Discussion about harmonic extension and the codimension 
problem}
If $\Gamma$ is the hyperbola $x_1x_2=1$, then for $\beta>1$ the bounded 
harmonic extensions to the upper half-plane of the functions 
$x\mapsto\e^{\imag m\pi x}$ and $x\mapsto\e^{\imag n\pi\beta/x}$, where 
$m,n$ range over the integers $\Z$,  fails to separate all the points of
the upper half-plane $\C_+:=\{z\in\C:\,\im z>0\}$. Indeed, if we consider
\[
z_1:=m+\imag\sqrt{\frac{\beta}{mn}-1},\quad 
z_2:=-m+i\sqrt{\frac{\beta}{mn}-1},\quad\mbox{where}\quad m,n\in\Z_+,
\quad mn<\beta,
\]
then, $f(z_1)=f(z_2)$ for every $f\in\mathcal{M}_\beta$, so that the
differences of Poisson kernels $P_{z_z}-P_{z_2}$ are elements in the
pre-annihilator space $\mathcal{M}_\beta^\bot$. If we use the Cauchy kernel in 
place of the Poisson kernel here we also obtain elements of the 
pre-annihilator. But there are only finitely many combinations of 
$m,n\in\Z_+$ with $mn<\beta$, which corresponds to finitely many differences
of Poisson or Cauchy kernels. This would lead us to suspect that that the 
pre-annihilator $\mathcal{M}_\beta^\bot$ might be finite-dimensional. Theorem
\ref{tpo} shows that this is far from being true. 
\end{subsection}

\begin{subsection}{Structure of the paper}
In Section \ref{ecdirac}, we take a closer look at the link between 
Theorem \ref{theop} and the Klein-Gordon and Dirac equations. In order to 
make the paper accessible to a wider audience, we present in Section 
\ref{preliminaries} the elementary aspects of the theory of dynamical systems 
and the standard notation for Perron-Frobenius operators needed here.
In Section \ref{Branch}, we show how the theory of Perron-Frobenius operators 
is the natural tool to analyze Heisenberg uniqueness pairs for the hyperbola
$\Gamma$ and for one of its branches $\Gamma_+$. In particular, the famous 
Gauss-Kuzmin-Wirsing operator corresponds to critical density case for 
$\Gamma_+$. In Section \ref{Moreabout}, we state some more involved results 
of the theory of Perron-Frobenius operators which are needed later on. 
In Section \ref{notation}, we study the structure of the pre-annihilator space 
$\mathcal{M}_\beta^\bot$ associated with $\Gamma$. In Section \ref{euimf}, 
we show the existence and uniqueness of absolutely continuous invariant 
measure for certain transformations acting on the interval $[-1,1]$. 
This is the key point in the proof of Theorem \ref{tpo}, presented in 
Section \ref{pruebas}. We end Section \ref{pruebas} by sketching the proof 
of Theorem \ref{tpo1001}, which turns out to be much simpler than that of 
Theorem \ref{tpo}. Finally, in Section \ref{applic}, we apply our results
to a problem involving the linear span of powers of two atomic singular 
inner functions in the Hardy space of the unit disk. 
In conclusion, we can say that the study of Heisenberg 
uniqueness pairs related to the Klein-Gordon equation leads to new and 
interesting problems involving Perron-Frobenius operators.
\end{subsection}

\begin{subsection}{Acknowledgements}
We thank Michael Benedicks for enlightening discussions on Perron-Frobenius
operators. 
\end{subsection}
\end{section}

\begin{section}{Further motivation. The Klein-Gordon and Dirac equations}
\label{ecdirac}

\begin{subsection}{The Dirac equation in three spatial dimensions}
In quantum mechanics the evolution of the position wave-function
$\psi$ associated to a physical system can be modelled by
certain partial differential equations (PDE). 
According to the theory of spin, in the general setting, $\psi$ has four 
components,
\[
\psi=(\psi_1,\psi_2,\psi_3,\psi_4),
\]
which should be thought of as written in column form, where each 
$\psi_j=\psi_j(t,x_1,x_2,x_3)$ is a mapping between an open set in $\R^4$ 
and a prescribed Hilbert space. Thus, these PDE's have to be understood 
as a system of equations for four separate wave-functions. The necessity of
working with multiple-component wave-functions was pointed out by
Pauli in order to understand the intrinsic angular momentum (spin)
of atoms. There is not a general equation whose solutions reflect
faithfully the evolution of a given system from a relativistic point
of view. Depending on the features of the system one must choose one
or another type of equation. For instance, for a relativistic spin-0
particle with rest-mass $m_0$ we have the Klein-Gordon equation.
Written in natural units it takes the form,
\[
\big(\partial_t^2 -\partial_{x_1}^2-\partial_{x_2}^2-\partial_{x_3}^2
+m_0^2\big)\psi =0.
\]
Another example, perhaps the most important in this context, is the
Dirac equation. It is used to describe the wave-function of the
electron, although it remains valid when applied to a general
relativistic spin-$\frac12$ particle. In natural units it takes the form,
\[
\big(-\imag\gamma^0\partial_t-\imag\gamma^1\partial_{x_1}
-\imag\gamma^2\partial_{x_2}-\imag\gamma^3\partial_{x_3}+m_0\big)\psi=0,
\]
commonly abbreviated as $(-\imag\not\!\partial+m_0)\psi=0$, where
$\gamma^0,\gamma^1,\gamma^2,\gamma^3$, the Dirac matrices, are the
$4\times 4$ matrices given by
\[
\gamma^0=\left(
             \begin{array}{cccc}
               1 & 0 & 0 & 0 \\
               0 & 1 & 0 & 0 \\
               0 & 0 & -1 & 0 \\
               0 & 0 & 0 & -1 \\
             \end{array}
           \right),\,\,\,
           \gamma^1=\left(
             \begin{array}{cccc}
               0 & 0 & 0 & 1 \\
               0 & 0 & 1 & 0 \\
               0 & -1 & 0 & 0 \\
               -1 & 0 & 0 & 0 \\
             \end{array}
           \right),
\]
\[
           \gamma^2=\left(
             \begin{array}{cccc}
               0 & 0 & 0 & -\imag \\
               0 & 0 & \imag & 0 \\
               0 & \imag & 0 & 0 \\
               -\imag & 0 & 0 & 0 \\
             \end{array}
           \right),\,\,\,
           \gamma^3=\left(
             \begin{array}{cccc}
               0 & 0 & 1 & 0 \\
               0 & 0 & 0 & -1 \\
               -1 & 0 & 0 & 0 \\
               0 & 1 & 0 & 0 \\
             \end{array}
           \right).
\]
The algebraic properties of the matrices $\gamma^0,\gamma^1,\gamma^2,\gamma^3$
allow us to obtain the following factorization of the Klein-Gordon equation:
\[
\big(\partial_t^2-\partial_{x_1}^2-\partial_{x_2}^2-\partial_{x_3}^2
+m_0^2\big)
\psi=(\imag\not\!\partial
+m_0)(-\imag\not\!\partial +m_0)\psi=0.
\]
Hence a solution to the Dirac equation is always a solution to the 
Klein-Gordon equation. The converse statement is not true.
\end{subsection}

\begin{subsection}{The Dirac equation in one spatial dimension}

\noindent As before, let $\Gamma$ be the hyperbola $x_1x_2=1$, and suppose 
$\mu\in\mathrm{AC}(\Gamma;\Lambda_{\alpha,\beta})$ for some positive reals
$\alpha,\beta$. Then, in view of \eqref{ecdp}, the Fourier transform 
$\widehat\mu$ given by \eqref{eq-FT} solves the partial differential equation
\[
(\partial_{x_1}\partial_{x_2}+\pi^2)\widehat\mu=0
\]
in the sense of distribution theory. If we write
$\psi(t,x):=\widehat{\mu}\left(\frac12(t+x),\frac12(t-x)\right)$, then
$\psi$ solves the one-dimensional Klein-Gordon equation for a particle of
mass $\pi$,
\begin{equation}
\label{kleina}
(\partial_t^2-\partial_x^2+\pi^2)\psi=0.
\end{equation}
Theorem \ref{theop} asserts that if $\alpha\beta>1$, there exists an 
infinite-dimensional space of solutions $\psi$ to \eqref{kleina} of the 
given form, subject to the condition of vanishing on
\[
\Lambda'_{\alpha,\beta}=\big\{(m\alpha,m\alpha)\in\R^2:n\in\Z\big\}\cup
\big\{(n\beta,-n\beta)\in\R^2:n\in\Z\big\}.
\]
The corresponding Dirac equation in this context is
\begin{equation}
\label{diraca}
\left(-\imag\,\sigma^0 \partial_t-\imag\,\sigma^1\partial_x+\pi\right)\psi=0,
\end{equation}
where $\sigma^0,\sigma^1$ are the $2\times2$ matrices given by 
\[
\sigma^0=\left(
             \begin{array}{cc}
               1 & 0 \\
               0 & -1 \\
             \end{array}
           \right)\quad\mbox{and}\quad\sigma^1=\left(
             \begin{array}{cc}
               0 & 1 \\
               -1 & 0 \\
             \end{array}
           \right).
\]
Here, $\psi=(\psi_1,\psi_2)$ in column form, and \eqref{diraca} may be 
written out more explicitly as the system
\begin{equation}
\label{bscoy}
\begin{cases}
-\imag\partial_t\psi_1-\imag\partial_x\psi_2+\pi\psi_1=0,&\\
\imag\partial_t\psi_2+\imag\partial_x\psi_1+\pi\psi_2=0.&
\end{cases}
\end{equation}
The question pops up whether the Dirac equation \eqref{diraca}
has an infinite-dimensional space of solutions $\psi=(\psi_1,\psi_2)$
that vanish on $\Lambda'_{\alpha,\beta}$ for $\alpha\beta>1$. As both
$\psi_1,\psi_2$ automatically solve the Klein-Gordon equation (this is a 
consequence of the factorization we mentioned previously in the context of
three spatial dimensions), the natural requirement
is that both $\psi_1,\psi_2$ are Fourier transform of measures in
$\mathrm{AC}(\Gamma',\Lambda'_{\alpha,\beta})$. Here, 
$\Gamma'$ is the hyperbola $t^2=x^2+1$, which corresponds to $\Gamma$ after
the change of variables. From the assumptions made on $\psi_1,\psi_2$, 
we have that
\[
\psi_j(t,x)=\int_{-\infty}^{+\infty}
f_j(\tpar)\,\e^{\imag\frac12\pi[\tpar(t+x)+\tpar^{-1}(t-x)]}
\diff\tpar,\qquad j=1,2,
\]
where $f_1,f_2$ belong to $\mathcal{M}_\beta^\bot$ (this subspace of $L^1(\R)$
is defined by \eqref{aniquilator}). Note that in the last step, we tacitly 
imposed the normalizing assumption that $\alpha=1$. 
As we implement this representation of $\psi_1,\psi_2$ into 
\eqref{bscoy}, we find that
\begin{equation*}
\int_{-\infty}^{+\infty}
\left\{(\tpar+\tpar^{-1}+2)f_1(\tpar)+(\tpar-\tpar^{-1})f_2(\tpar)\right\}
\e^{\imag\frac12\pi[\tpar(t+x)+\tpar^{-1}(t-x)]}\diff\tpar=0
\end{equation*}
and
\begin{equation*}
\int_{-\infty}^{+\infty}
\left\{(\tpar+\tpar^{-1}-2)f_2(\tpar)+(\tpar-\tpar^{-1})f_1(\tpar)\right\}
\e^{\imag\frac12\pi[\tpar(t+x)+\tpar^{-1}(t-x)]}\diff\tpar=0.
\end{equation*}
As we plug in $t=x$, we see from the uniqueness theorem for the Fourier 
transform that the above two equations are equivalent to having
\[
(\tpar+\tpar^{-1}+2)f_1(\tpar) +(\tpar-\tpar^{-1})f_2(\tpar)=0,\qquad 
\tpar\in\R,
\]
and
\[
(\tpar-\tpar^{-1})f_1(\tpar)+(\tpar+\tpar^{-1}-2)f_2(\tpar)=0,\qquad\tpar\in\R,
\]
in the almost-everywhere sense. These requirements are compatible, as each one
corresponds to having
\[
f_2(\tpar)=\frac{1+\tpar}{1-\tpar}\,f_1(\tpar),\qquad \tpar\in\R.
\]
This means that we have reduced the study of the dimension of the space of
solutions to the Dirac equation \eqref{diraca} subject to vanishing on  
$\Lambda'_{\alpha,\beta}$ (with $\alpha=1$) plus the condition in terms of
the Fourier transform to simply analyzing the dimension of the space 
\[
\big\{f\in\mathcal{M}_\beta^\bot:\,(1+x)(1-x)^{-1}f(x)\,\,\,
\text{is in}\,\,\,\mathcal{M}_\beta^\bot\big\}.
\]
To answer this dimension question we would need to better understand the 
structure of the pre-annihilator space $\mathcal{M}_\beta^\bot$.
\end{subsection}
\end{section}

\begin{section}{Perron-Frobenius operators}
\label{preliminaries}

\begin{subsection}{Dynamical systems}
The theory of dynamical systems deals with the time evolution of a
system of points under a fixed change rule. An important feature of
a dynamical system is its attractors, sets of points towards which
the points of the system converge. A dynamical system is a
four-tuple $(I,\sigmaalg,\mu,\tau)$, where $(I,\sigmaalg,\mu)$ is a
measure space and $\tau:I\longrightarrow I$ is a measurable
transformation. The measure $\mu$ is always positive and $\sigma$-finite; 
if it has finite total mass we renormalize and assume that the mass is $1$,
so that $\mu$ becomes a probability measure. We denote by $\tau^0$ the 
identity map and write $\tau^n=\tau^{n-1}\circ\tau$ for $n=1,2,3,\ldots$. 
The evolution of a point $x\in I$ is described by its {\em orbit} under 
$\tau$, i.e., the sequence
\[
\{\tau^n(x)\}_{n=0}^{+\infty}.
\]
In a concrete situation, the actual expression for the iterates $\tau^n$ 
tends to explode already for rather modest values of $n$, which makes it 
extremely difficult to extract substantial information based a direct 
approach. The most convenient approach is then the measure-theoretic one 
based on Perron-Frobenius operators. If we have a random variable 
$X:I\longrightarrow \R$ distributed according to a density $\rho$, then the 
random variable
$X\circ\tau$ will be distributed according to a new density, which
is denoted by $\Pop_\tau\rho$. Instead of the orbits
$\tau^n(x)$, we focus on the sequence of density functions
\[
\{\Pop_\tau^n \rho\}_{n=0}^{+\infty}.
\]
The key point here is that, while $\tau$ may be nonlinear and discontinuous, 
the operator $\Pop_\tau$ is linear and bounded on the space 
$L^1(I,\sigmaalg,\mu)$ of integrable functions on $I$. The operator 
$\Pop_\tau$ is known as the Perron-Frobenius operator associated to 
the transformation $\tau$. It turns out that in most situations the 
sequence of density functions $\Pop^n_\tau\rho$ converges to  certain 
densities of measures on $I$ that provide valuable information about the
attractors of the system, which are known as {\em invariant measures}. More
precisely, a $\sigma$-finite Borel measure $\nu$ on $I$ is said to be 
{\em invariant} under $\tau$ if $\nu(\tau^{-1}(A))=\nu(A)$, for every 
$A\in\sigmaalg$. The densities of the $\mu$-absolutely continuous invariant 
measures can be recovered as eigenfunctions of the Perron-Frobenius operator 
corresponding to the eigenvalue $\lambda=1$. Perron-Frobenius operators 
appear in many
branches of pure and applied mathematics, such as in stochastic processes,
statistical mechanics, resonances, ordinary differential equations,
thermodynamics, diffusion problems, positive matrices, and algorithms 
associated with continued fractions expansions. For a backgound on 
Perron-Frobenius operators, we refer to, e.g., see \cite{blank}, \cite{boy},
\cite{driebe}.  In this work, we shall see how Perron-Frobenius
operators are intimately related to the Heisenberg uniqueness pairs associated
with the hyperbola $x_1x_2=1$. This leads to new and interesting questions 
concerning this important class of operators.
\end{subsection}

\begin{subsection}{Perron-Frobenius operators on bounded intervals}
\label{subsec-PFBI}
In  our situation, the  dynamical systems involved are of the form
$(I,\borelalg_I,m,\tau)$, where $I$ is a closed bounded interval of the real 
line, $m$ is the Lebesgue measure defined on $\borelalg_I$, the Borel 
$\sigma$-algebra of $I$, and $\tau$ denotes a measurable map from  $I$ into 
itself. For $1\leq p<+\infty$, the Banach  space $L^p(I)$ consists of those 
measurable complex-valued (equivalence classes of) functions $f$ defined on 
$I$ for which the norm
\[
\|f\|_{L^p(I)}^p=\int_I|f|^p\, \diff m
\]
is finite. The space $L^\infty(I)$ consists of the essentially bounded 
measurable complex-valued functions $f$ supplied with the essential supremum 
norm. We shall use the following standard bilinear dual action:
\begin{equation}
\label{eq-duality}
\langle f,g\rangle_E:=\int_{E}fg\,\diff m,
\end{equation}
provided $f,g$ are Borel measurable, and $fg\in L^1(E)$. Here, $E\subset\R$
is a Borel set with positive linear measure: $m(E)>0$. For instance, if
$f\in L^1(E)$ and $g\in L^\infty(E)$, the dual action is well-defined.
When needed, we shall think of functions in $f\in L^p(E)$ as extended to
all of $\R$ by setting them equal to $0$ off $E$. 

We shall need the following concepts.

\begin{defn}
The map $\tau:I\to I$ is said to be a {\em filling $C^2$-smooth piecewise 
monotonic transformation} if there exists a countable collection of
pairwise disjoint open intervals $\{I_u\}_{u\in\Uset}$, such that

\noindent
{$(i)$} the set $I\setminus\cup\{I_u:\,u\in\Uset\}$ has
linear Lebesgue measure $0$,

\noindent
{$(ii)$} for any $u\in\Uset$, the restriction of $\tau$ to 
$I_u$ is strictly monotonic and extends to a $C^2$-smooth function on the
closure of $I_u$, denoted $\tau_u$, and $\tau_u'\ne0$ holds in the interior of 
$I_u$,

\noindent
{$(iii)$} for every $u\in\Uset$, $\tau_u$ maps the closure of 
$I_u$ onto $I$.
\label{def-fpw}
\end{defn}

\begin{defn} If, in the setting of Definition \ref{def-fpw}, all conditions are
fulfilled, save that $(iii)$  is replaced by the weaker condition $(iii')$ 
below, we say that $\tau$ is a {\em partially filling $C^2$-smooth piecewise 
monotonic transformation}:

\noindent{$(iii')$} there exists a $\delta>0$, such that for every
$u\in\Uset$, the length of the interval $\tau(I_u)$ is $\ge\delta$.
\label{def-pfpw}
\end{defn}

In the context of the above two definitions, each intervals $I_u$ is known
as a {\em fundamental interval}, and $\tau_u$ is said to be a {\em branch}.
It is an important observation that each iterate $\tau^n$, with 
$n=1,2,3,\ldots$, has the same basic structure as the transformation $\tau$ 
itself. The fundamental intervals associated with $\tau^n$ are given by
\[
I_{(u_1,\ldots,u_n)}^n=\big\{x\in I:x\in I_{u_1},\tau(x)\in 
I_{u_2},\ldots,\tau^{n-1}(x)\in I_{u_n}\big\},\qquad 
(u_1,\ldots,u_n)\in\Uset^n_\sharp,
\]
where $\Uset^n_\sharp$ consists of those elements 
$(u_1,\ldots,u_n)\in\Uset^n$ such that the above interval 
$I_{(u_1,\ldots,u_n)}^n$ becomes non-empty.
The corresponding branch on $I_{(u_1,...,u_n)}^n$ is denoted
$\tau_{(u_1,\ldots,u_n)}^n=\tau_{u_n}\circ\cdots\circ\tau_{u_1}$.

The {\em Koopman operator} associated with $\tau$ is the
composition operator which acts on $L^\infty(I)$ by the formula 
$\Cop_\tau g=g\circ\tau$. Clearly, $\Cop_\tau$ is linear 
and norm-contractive on $L^\infty(I)$. bounded. The
Perron-Frobenius operator $\Pop_\tau:\,L^1(I)\to L^1(I)$ 
associated with $\tau$ is just the pre-adjoint of $\Cop_\tau$. 
Therefore, $\Pop_\tau$ is a norm contraction on $L^1(I)$ with
\begin{equation}\label{preadjunto}
\langle\Pop_\tau f,g\rangle_I=\langle f,\Cop_\tau g\rangle_I,
\qquad f\in L^1(I),\,\,\, g\in L^\infty(I).
\end{equation}
It is immediate from \eqref{preadjunto} that an absolutely continuous measure 
$\diff\mu_f=f\,\diff m$ is has the invariance property
\[
\mu_f(\tau^{-1}(A))=\mu(A)\quad\text{for all}\,\,\,A\in\borelalg_I
\]
if and only if
\begin{equation}\label{invariantef}
\Pop_\tau f=f.
\end{equation}
It is clear that $\Cop_\tau^n=\Cop_{\tau^n}$, so, by duality
we have that
\begin{equation}
\label{iteradada}
\Pop_\tau^n=\Pop_{\tau^n},\qquad n=1,2,3,\ldots.
\end{equation}
Using \eqref{iteradada}, we find that
\begin{equation}\label{retardo}
(\Pop_\tau f)(x)=\sum_{u\in\Uset}
J_u(x)f(\tau_u^{-1}(x)),\qquad n=1,2,3,\ldots,
\end{equation}
where $J_u\ge0$ is the function on $I$ that equals
$|(\tau_u^{-1})'|$ on $\tau(I_u)$ and vanishes elsewhere.
The map $J_u$ is well defined, since $\tau$ is piecewise strictly
monotonic. By \eqref{retardo} we see that $\Pop_\tau f\geq 0$
for $f\geq 0$ and $\Vert \Pop f \Vert_1=\Vert f \Vert_1$. 
As $\Pop_\tau$ acts contractively on $L^1(I)$, its spectrum 
$\sigma(\Pop_\tau)$ is contained in the closed unit disk $\bar\D$.
\end{subsection}
\end{section}

\begin{section}{Perron-Frobenius operators for Gauss-type maps and invariant
measures}
\label{Branch}

\begin{subsection}{The Gauss-type maps and the corresponding Perron-Frobenius
operators}
\label{subsec-GTM1}
For $t\in\R$, let $\{t\}_1$ be the number in the interval $[0,1[$ such that
$t-\{t\}_1\in\Z$. We also need the expression $\{t\}_2$, which is in the 
interval $]\!-\!1,1]$ and is uniquely determined by the requirement $t-\{t\}_2
\in2\Z$. Next, for $0<\gamma<+\infty$, we consider the Perron-Frobenius 
operator for the transformation $\sigmat_\gamma(x):=\{\gamma/x\}_1$ on the 
unit interval $[0,1[$ (we put $\sigmat_\gamma(0):=0$ to define the map at the 
origin), which is given by 
\begin{equation}
\label{gauss}
\Pop_{\sigmat_\gamma} f(x)=\sum_{j=1}^{+\infty}\frac{\gamma}{(j+x)^2} 
f\left(\frac{\gamma}{j+x}\right), \qquad x\in [0,1[,
\end{equation} 
for $f \in L^1([0,1[)$, with the understanding that $f$ vanishes off $[0,1[$. 
It is easy to see that the eigenfunction equation
\[
\Pop_{\sigmat_\gamma} f=\lambda f,\qquad |\lambda|=1,
\]
fails to have a solution $f$ in $L^1([0,1[)$ for $0<\gamma<1$. In the
case $\gamma=1$, $\Pop_{\sigmat_1}$ is the famous {\em Gauss-Kuzmin-Wirsing} 
operator, which is connected with the continued fraction algorithm. 
It is known that $\Pop_{\sigmat_1}f=\lambda f$ with $|\lambda|=1$ has a 
non-trivial solution only for $\lambda=1$, in which case the solution $f$ 
is unique (up to a scalar multiple). These observations are basic in
the proof of Theorem \ref{theobranch} (or, which is the same, Theorem 
\ref{tpo100}); the natural parameter choices are $\alpha=2$ and 
$\beta=2\gamma$.

For $0<\beta<+\infty$, we may instead consider the Perron-Frobenius operator 
for the transformation $\tau_\beta(x):=\{-\beta/x\}_2$ on the interval 
$]\!-\!1,1]$ (we put $\tau_\beta(0):=0$ to define the map at the origin), 
which is given by
\begin{equation}
\label{eq-gauss-var}
\Pop_{\tau_\beta}f(x)=\sum_{j\in\Z^\times} \frac{\beta}{(2j-x)^2}
f\left(\frac{\beta}{2j-x}\right),\qquad x\in ]\!-\!1,1], 
\end{equation}
for $f\in L^1(]\!-\!1,1])$, with the understanding that $f$ vanishes off 
$]\!-\!1,1]$. 
Here, we use the standard notation 
$\Z^\times:=\Z\setminus\{0\}$. A rather elementary argument shows that the 
eigenfunction equation
\[
\Pop_{\tau_\beta} f=\lambda f,\qquad |\lambda|=1,
\]
fails to have a solution $f$ in $L^1(]\!-\!1,1])$ for $0<\beta<1$. 
For $\beta=1$,
the transformation $\tau_1(x)=\{-1/x\}_2$ is related to the continued fraction
algorithm with even partial quotients, cf. \cite{Kraa}, \cite{Brigit}. The
map $\tau_1(x)=\{-1/x\}_2$ has an indifferent fixed point at $1$. This entails
that the invariant absolutely continuous density has infinite total mass;
in this case, the density is given explicitly by $(1-x^2)^{-1}$. Using some
ergodicity properties, it is easy to show that the equation 
$\Pop_{\tau_1}f=\lambda f$ fails to have solutions $f$ in $L^1(]\!-\!1,1])$
for all $\lambda\in\C$ with $|\lambda|=1$. These observations are basic in
the proof of Theorem \ref{thh} ( or, which is the same, Theorem \ref{thhh}).
\end{subsection}

\begin{subsection}{Discrete and singular invariant measures for the 
Gauss map}
\noindent It is well-known that the Gauss map $\sigmat_1(x):=\{1/x\}_1$
on the unit interval $[0,1[$ [with $\sigmat_1(0):=0$] 
has infinitely many essentially different invariant measures. However, 
up to a constant multiple, there is only
one that is absolutely continuous: $(1+x)^{-1}\diff x$. 
As for the discrete bounded invariant measures for $\sigmat_1$ there is an easy
description. We write $\delta_a$ for the Dirac measure at point $a$. We shall 
need the set of fixed points of iterates of the Gauss map $\sigmat_1$:
\[
\setS_k:=\{a\in[0,1[:\,\sigmat_1^k(a)=a\}, \qquad\setS_\infty:=
\bigcup_{k=1}^{+\infty}\setS_k.
\]
For $a\in\setS_\infty$, there exists a minimal $k\ge1$ such that 
$\sigmat_1^k(a)=a$; we write $k(a)$ for this $k$. We put
\[
\rho_a:=\frac{1}{k(a)}\sum_{j=0}^{k(a)-1}\delta_{\sigmat_1^j(a)},\qquad 
a\in\setS_\infty.
\] 

\begin{thm}
\label{surdsos} 
Let $\mu$ be a discrete bounded invariant measure for the Gauss map 
$\sigmat_1$. 
Then there is a function $\xi:\setS_\infty\to\C$ with
$\sum_{a\in\setS_\infty}|\xi(a)|<+\infty$, such that
\begin{equation}\label{discrto}
\mu=\sum_{a\in\setS_\infty}\xi(a)\,\rho_a.
\end{equation}
\end{thm}

\begin{proof}
That the measure $\mu$ is invariant means that 
\[
\int_{[0,1[}f\circ\sigmat_1(x)\diff\mu(x)=
\int_{[0,1[}f(x)\,\diff\mu(x)
\]
for every $f$ integrable with respect to $|\mu|$. It is trivial to check that
all measures of the given form are invariant. 
In the other direction, it is well-known that every discrete invariant 
measure $\mu$ may be decomposed into irreducible (ergodic) parts. 
We just need to show that up to a multiplicative constant, each irreducible 
part is of the form $\rho_a$.
So, let $\rho$ be an ergodic discrete invariant probability measure on 
$[0,1[$. Let $E\subset[0,1[$ be the minimal countable set which carries the 
mass of $\rho$. By the Birkhoff-Khinchin Ergodic Theorem, we have for each 
$x_0\in E$ that
\[
\frac{1}{n}\sum_{j=0}^{n-1} f(\sigmat_1^j(x_0))\to\int_{E}f\diff\rho\quad
\text{as}\,\,\,n\to+\infty.
\]
In particular, if we let $f$ equal the characteristic function of the one-point
set $\{x_0\}$, and observe that from the minimality of $E$, we have 
$\rho(\{x_0\})>0$, we find that approximately a $\rho(\{x_0\})$ proportion 
of the time on the interval $0\le j\le n-1$, we have 
$\sigmat_\gamma^j(x_0)=x_0$. 
This means that $x_0\in\setS_\infty$.
We may also conclude by picking other functions $f$ that $E$ equals the orbit 
of $x_0$:
\[
E=\{\sigmat_\gamma^j(x_0):\,0\le j\le k(x_0)-1\}.
\]
By ergodicity, each point of $E$ must have equal mass, so that 
$\rho=\rho_a$ with $a=x_0$.
This completes the proof.
\end{proof}

\begin{rem} $(a)$ The proof of Theorem \ref{surdsos} did not really use 
the fact that we are dealing with the Gauss map. In particular, the 
corresponding assertion holds for the transformation $\tau_1(x)=\{-1/x\}_2$ 
in place of the Gauss map.

\noindent{$(b)$} The set $\setS_\infty$ consists of the quadratic surds plus
the origin.

\noindent{$(c)$} The discrete measures provided by Theorem \ref{surdsos} 
above extend in a natural fashion to the positive half-axis $\R_+$. These
extensions lift to the hyperbola branch $\Gamma_+$ where $x_1x_2=1$ and 
$x_1>0$, and the lifted measures on $\Gamma_+$ have Fourier transforms that
vanish on the lattice-cross $\Lambda_{\alpha,\beta}$ with $\alpha=\beta=2$.
Indeed, this is the only way to obtain such discrete measures on $\Gamma_+$.
\end{rem}
\end{subsection}

\begin{subsection}{The Minkowski measure}
The most studied  singular continuous measure for the
Gauss-Kuzmin-Wirsing operator  is the Minkowski measure. It belongs to a family
of singular probability measures which we call Markovian measures (see
below).
The Minkowski question mark function was first introduced by Minkowski
in 1904. Let $\{ a_n(x)\}_{n=1}^{+\infty}$ be the sequence of positive integers
in the continued fraction expansion of $x$. Salem \cite{Salem} proved that 
the question mark function can be defined by
\[
?(x)= \sum_{j=1}^{+\infty} \frac{2(-1)^{j+1}}{2^{a_1(x)+a_2(x)+\cdots
+a_j(x)}}, \quad \quad  0\leq x \leq 1,
\]
where the series is finite for rational $x$. Then $?$ is a strictly 
increasing continuous singular function. It
takes the rational numbers into the dyadic numbers and the
quadratic surds into the rationals. The Riemann-Stieltjes measure
$\diff?$ on $[0,1]$, is then a singular continuous probability
measure, which can be shown to be invariant for the Gauss map 
$\sigmat_1(x)=\{1/x\}_1$.
The numbers $a_j(x)\in\Z_+$ are the successive remainders which we throw
away as we iterate the Gauss map at a point $x$. Let us say that a 
probability measure $\mu$ on $[0,1]$ is {\em Markovian} with respect to the 
Gauss map if there are numbers $q(j)$ with $0\le q(j)<1$ and
\[
\sum_{j=1}^{+\infty}q(j)=1,
\]
such that the $\mu$-mass of the ``cylinder set'' 
\[
\big\{x\in[0,1[:\,a_j(x)=b_j\,\,\,\text{for}\,\,\, j=1,\ldots,k\big\}
\]
equals
\[
\prod_{j=1}^{k}q(b_j).
\]
Then the Minkowski measure is Markovian with $q(j)=2^{-j}$.
Problem (d) in \cite{hh} could be settled in the negative if we could find
a singular continuous measure on $\Gamma$, the Fourier transform of which
tends to zero at infinity, while it vanishes along $\Lambda_{\alpha,\beta}$
with $\alpha=\beta=1$. 
A perhaps easier task is to find a singular invariant measure on $[0,1]$
with respect to the Gauss map such that the Fourier transform tends to zero
at infinity (measures whose Fourier transforms decay to $0$ at infinity are
called {\em Rajchman measures}). The Markovian measures are all invariant 
and singular continuous. E.g., the Minkowski measure is of this type. But it is
not known if it is Rajchman. Indeed, this question is a well-known problem 
raised by Salem \cite{Salem}.
\end{subsection}
\end{section}

\begin{section}{Further properties of Perron-Frobenius operators}
\label{Moreabout}

\begin{subsection}{The spectral decomposition of Perron-Frobenius operators}
Let $I$ be a closed bounded interval. 
The total variation of a complex-valued function
$h:I\longrightarrow\C$ is
\[
\mbox{var}_I(h)=\sup\bigg\{\sum_{j=1}^{n-1}|h(t_{j+1})-h(t_{j})|\bigg\},
\]
where the supremum is taken over all $t_1,\ldots,t_n\in I$ with
$t_1<\ldots<t_n$.  The function
$h$ is said to be of {\em bounded variation} when $\mbox{var}_I(h)<+\infty$.
We will denote by  $\BV(I)$ the subspace of $L^1(I)$ functions, which
has a representative of bounded variation. The space $\BV(I)$ becomes
 a Banach space when supplied, e.g., with the norm
\begin{equation}\label{normada}
\|h\|_{\BV}=\|h\|_1+\inf_{\tilde{h}\sim
h}\mbox{var}_I(\tilde{h}), \quad \quad h \in\BV(I),
\end{equation}
where the infimum is taken over the elements in the equivalence class of $h$
(so that $\tilde h=h$ except on a Lebesgue null set). 
It is well-known that for each $h\in\BV(I)$  there is a
right-continuous function in the class of $h$ where the infimum in
the definition of $\|h\|_{\BV}$ in \eqref{normada} is attained. In
particular, $BV(I)$ is a subspace of $L^\infty(I)$, see \cite{Ios}. A perhaps
more precise description of $\BV(I)$ is that these are the primitive 
functions of the finite complex-valued Borel measures on $I$. 

Let $\T=\{z\in\C:|z|=1\}$ denote the unit circle in $\C$ and let 
$\Spec_{\mathrm{p}}(\Pop_\tau)$ denote the point spectrum of
$\Pop_\tau$, where the Perron-Frobenius operator $\Pop_\tau$ 
is thought of as acting on $L^1(I)$. As  a consequence of the
Ionescu-Tulcea and Marinescu theorem, see \cite{boy} and \cite{Ios}, Section 
5.3, the following spectral decomposition holds for $\Pop_\tau$.
We recall the notions of filling and partially filling $C^2$-smooth monotonic
transformations in Definitions \ref{def-fpw} and \ref{def-pfpw}.

\begin{thml}\label{tf}
Suppose $\tau:I\to I$ is a partially filling $C^2$-smooth piecewise monotonic 
trans\-formation with the follow\-ing properties $(i)$--$(ii)$:
\\
\noindent
$(i)$ {\rm[uniform expansiveness]}  There exist and integer $m\ge1$ and
a positive real $\epsilon$ such that \hfill\break
$|(\tau^{m})'(x)|\ge1+\epsilon$ for all 
$x\in\cup \{I_u:\,u\in\Uset^m_\sharp\}$.

\noindent
$(ii)$ {\rm[second derivative condition]} 
There exists a positive constant $M$ such that
$|\tau''(x)|\le M|\tau'(x)|^2$ for all 
$x\in\cup\{I_u:\,u\in\Uset\}$.
\\
\noindent
Then $\Lambda_\tau:=
\Spec_{\mathrm{p}}(\Pop_\tau)\cap\T$ is finite and non-empty, say 
$\Lambda_\tau=\{\lambda_1,\ldots,\lambda_p\}$.
Here, one of the eigenvalues is the point $1$, say $\lambda_1=1$. 
If $E_i$ denotes the eigenspace of
$\Pop_\tau$ corresponding to $\lambda_i$, then $E_i$ is
finite-dimensional and $E_i\subset \BV(I)$. In addition, we have
\[
\Pop_\tau^{n}h=\sum_{i=1}^p\lambda_i^n 
\Pop_{\tau,i}h+\Zop_\tau^{n}h,\qquad h\in L^1(I),\,\,\,n=1,2,3,\ldots,
\]
where the operators $\Pop_{\tau,i}$ are the projections onto $E_i$, and
the operator $\Zop_\tau$ acts boundedly on on $L^1(I)$ as well as on
$\BV(I)$. Moreover, the spectrum of $\Zop_\tau$ as an operator
acting on $\BV(I)$ is contained in the open unit disk $\D$; i.e.,
$\Zop_\tau$ acting on $\BV(I)$ has spectral radius $<1$. 
\end{thml}

\begin{rem}
\label{rem-5.1}
$(a)$ It is a by-product of Theorem \ref{tf} that $\Pop_\tau$ acts
boundedly on $\BV(I)$. In fact, the way things work is that this 
rather elementary observation is the beginning of the analysis that leads
up to Theorem \ref{tf}.

\noindent{$(b)$} Since $1$ is an eigenvalue of $\Pop_\tau$, the
corresponding eigenfunction (which is then in $\BV(I)$) is the density
for an invariant measure. If there are several such eigenfunction for 
$\lambda_1=1$, then one of them is $\ge0$, which we can normalize so that we 
get an absolutely continuous invariant probability measure with density in 
$\BV(I)$; compare with the proof of Theorem \ref{lemff}. 
 
\noindent{$(c)$} By a theorem of Rota (see \cite{boy}, \cite{Sch}), 
the exterior eigenvalues $\{\lambda_1,\ldots,\lambda_p\}$ form 
a finite union of cyclic subgroups of $\T$. In particular, they are all roots 
of unity, that is, $\lambda_i^N=1$ holds for all $i=1,\ldots,p$, for some big 
enough positive integer $N$.

\noindent{$(d)$} The formulation of the Ionescu-Tulcea and Marinescu
theorem in \cite{Ios} initially assumes that $\tau$ is ``filling'', but 
it is later remarked that the theorem holds for ``partially filling''
transformations (cf. Definitions \ref{def-fpw} and \ref{def-pfpw});
see \cite{Ios}, p. 214, and also \cite{bugiel} and \cite{CFS}, 
p. 169.  

\noindent{$(e)$} When considered as an operator on $L^1(I)$, the 
Perron-Frobenius operator $\Pop_{\tau}$ will usually have eigenvalues at 
all points of the open disk, with eigenfunctions in $L^\infty([-1,1])$ (cf.
\cite{Kel}). 
\end{rem}

\begin{rem} \label{rem-5.2}
From the presentation in \cite{Ios}, Section 5.3, it is clear that if $\tau$
is ``filling'', we have a stronger assertion in Theorem \ref{tf}: 
$\lambda_1=1$ is the only eigenvalue of $\Pop_\tau$ on $\T$,
and that the $\tau$-invariant absolutely continuous probability measure 
is unique, with a density that is bounded from above and below by two 
positive constants. Cf. also \cite{CFS}, p. 172, where it is shown that under 
the given assumptions, $\tau$ is {\em mixing}. 
We briefly outline the argument, following the presentation in \cite{Ios},
Section 5.3. We write $f_u:=\tau_u^{-1}:I\to \bar I_u$
for the inverse branches ($u\in\Uset$).
The assumptions $(i)$ and $(ii)$ of Theorem  correspond to the 
conditions $(\mathrm{E}_m)$ and (A) of \cite{Ios}, pp. 191--192. Next, by
\cite{Ios}, Proposition 5.3.3, we have that $(\mathrm{E}_m)$ forces 
$|f'_u(x)|$ to be uniformly bounded in $x\in I$ and $u\in\Uset$. 
In view of Condition (A) of \cite{Ios}, p. 192, we also know that 
$|f''_u(x)|$ is uniformly bounded in $x\in I$ and $u\in\Uset$. 
In particular, then, $f'_u$ is absolutely continuous for all 
$u\in\Uset$. Next, by \cite{Ios}, Proposition 5.3.4, we use this 
absolute continuity together with condition $(\mathrm{E}_m)$ of \cite{Ios},
p. 191, to see that condition (C) [R\'enyi's distortion estimate] holds as 
well. Next, by \cite{Ios}, Theorem 4.3.5, we use condition (C) to get that
there is a unique ergodic $\tau$-invariant absolutely continuous probability
measure, and that its density is bounded from above and below by positive 
constants. This means that there is only one eigenvalue, namely $1$. We
mention here that the condition (BV) of \cite{Ios}, p. 200 -- which requires
the sum of the variation of $f_u'$ over $u\in\Uset$ to be bounded 
-- is a trivial consequence of condition (A), as the sum of the variances of 
$f_u'$ over $u\in\Uset$ amounts to summing the lengths of the intervals $I_u$, 
which all add up to the length of $I$. 
\end{rem}


\end{subsection}

\begin{subsection}{The Folklore theorem (Adler's theorem)} 

We shall be interested in partially filling $C^2$-smooth monotonic 
transformations of a closed bounded interval $I$. In view of Remark 
\ref{rem-5.1} $(d)$, we can be assured that Theorem \ref{tf} holds also
in this more general situation. Moreover, Remark \ref{rem-5.1} $(b)$
tells us that there exists a $\tau$-invariant absolutely continuous 
probability measure, but it might not be unique. To get uniqueness, we need
to make stronger conditions on $\tau$. In this direction we have Adler's 
Theorem, also known as the Folklore theorem (see \cite{boy}).

\begin{thml}[\textrm{Adler's theorem}]
\label{adlt} 
Let $\tau$ be a partially filling $C^2$-smooth piecewise monotonic 
transformation with the following properties $(i)$--$(iv)$: 
\\
$(i)$ {\rm[uniform expansiveness]} There exist an integer $m\ge1$ and a 
positive real $\epsilon$ such that $|(\tau^m)'(x)|\ge1+\epsilon$ for all 
$x\in\cup\{I_u:\,u\in\Uset^n_\sharp\}$. 

\noindent
$(ii)$ {\rm[second derivative condition]} 
There exists a positive constant $M$ such that
$|\tau''(x)|\le M|\tau'(x)|^2$ for all 
$x\in\cup\{I_u:\,u\in\Uset\}$.

\noindent
$(iii)$ {\rm [Markov property 1]}
For every $u\in\Uset$ there is $n=n(u)\geq 1$ such that 
$\mathrm{clos}[\tau^n(I_u)]=I$.

\noindent
$(iv)$ {\rm [Markov property 2]} 
Whenever $\tau(I_u)\cap I_v\neq\emptyset$ holds for two indices 
$u,v\in\Uset$, then $\tau(I_u)\supset I_v$.
\medskip

\noindent
Then $\tau$ admits a unique absolutely continuous invariant probability 
measure $\diff\rho=\varrho\diff m$. Moreover, the density $\varrho$ is 
bounded from above and below by positive constants.
\end{thml}

\begin{rem}
$(a)$ 
A transformation $\tau$ satisfying $(iii)$--$(iv)$ above is said
to be a {\em Markov map}. 

\noindent{$(b)$} A well-known result which preceded Adler's theorem is the 
Lasota and Yorke theorem \cite{Las}. 
\end{rem}
\end{subsection}

\begin{subsection}{Dynamical properties of the Gauss-type maps}
\label{subsec-dyn}

We first consider the transformation $\tau_\beta$ of the interval 
$]\!-\!1,1]$, as defined by $\tau_\beta(0):=0$ and by
\begin{equation}
\label{ecuatotal}
\tau_\beta(x):=\left\{-\frac{\beta}{x}\right\}_2,\qquad x\ne0.
\end{equation}
Here, we recall that for $t\in\R$, $\{t\}_2$ denotes the uniquely determined 
number in $]\!-\!1,1]$ with with $t-\{t\}_2\in 2\Z$. We {\em restrict our
attention to $\beta>1$ only}. Let the index set $\Uset=\Uset_\beta$ be the 
subset of the non-zero integers $u$ for which the corresponding branch 
interval is non-empty: 
\begin{equation}
\label{eq-Iu}
I_u:=\bigg]\frac{\beta}{2u+1},\frac{\beta}{2u-1}\bigg[\cap\, 
]-1,1[\ne\emptyset.
\end{equation}
We put $u_0=u_0(\beta):=\frac{1}{2}(\beta-\{\beta\}_2)$, which is an 
integer $\ge1$. 
We note that if $\beta$ is an {\em odd integer}, then 
\begin{equation}
\label{eq-Iu1}
I_u=\bigg]\frac{\beta}{2u+1},\frac{\beta}{2u-1}\bigg[,\qquad u\in\Uset,
\end{equation}
and $\Uset$ consists of all $u\in\Z^\times$ with 
$|u|\ge\frac12(\beta+1)$. In this case, the ``filling'' requirement is
fulfilled: $\tau_\beta(I_u)=]-1,1[$ holds for all $u\in\Uset$. 
More generally, when $\beta$ is not an odd integer, then $\Uset$ consists of
all $u\in\Z^\times$ with $|u|\ge u_0$, and we have
\begin{equation}
I_u=\bigg]\frac{\beta}{2u+1},\frac{\beta}{2u-1}\bigg[,\qquad u\in\Uset
\setminus\{\pm u_0\}.
\label{eq-Iu2}
\end{equation}
so that 
\[
\tau_\beta(I_u)=]-1,1[,\qquad u\in\Uset
\setminus\{\pm u_0\}.
\]
We see that the deviation from the ``filling'' requirement is rather slight
(just two branches fail). 

Next, we quickly calculate the derivative of $\tau_\beta$:
\begin{equation}
\tau'_\beta(x)=\frac{\beta}{x^2}\ge\beta>1,\qquad x\in\cup
\{I_u:\,u\in\Uset\},
\label{eq-expand}
\end{equation}
so the uniform expansiveness condition is met already by $\tau_\beta$ (with
$m=1$). Moreover, 
\[
\frac{|\tau''_\beta(x)|}{|\tau'_\beta(x)|^2}\le \frac{2|x|}{\beta}\le 2,
\qquad x\in\cup\{I_u:\,u\in\Uset\},
\]
so we also have the second derivative control.
Unfortunately, we cannot rely on Theorem \ref{tf} to give us the uniqueness 
of the absolutely continuous invariant measure for $\tau_\beta$,  as it 
apparently allows for non-uniqueness (although Remark \ref{rem-5.2} says that 
in the ``filling'' case we really do have the needed uniqueness). We also
cannot rely on  Adler's theorem, as $\tau_\beta$ is not necessarily a Markov 
map. We will check that condition $(iii)$ of Adler's theorem holds for all
sufficiently big $n$, say $n\ge n(u)$. 

We are also interested in the Gauss-type map $\sigmat_\gamma$ of the interval 
$[0,1[$, as defined by $\sigmat_\gamma(0):=0$ and
\begin{equation}
\label{ecuatotal2}
\sigmat_\gamma(x):=\left\{\frac{\gamma}{x}\right\}_1,\qquad x\ne0.
\end{equation}
Here, we recall that $\{t\}_1$ is the fractional part of $t\in\R$, with 
values in $[0,1[$ and $t-\{t\}_1\in\Z$. 
We {\em restrict our attention to $\gamma>1$ only}. Let the index set 
$\Vset=\Vset_\gamma$ be the subset of the positive integers $v$ for which the 
corresponding branch interval is non-empty: 
\[
J_v:=\bigg]\frac{\gamma}{v+1},\frac{\gamma}{v}\bigg[\cap\, 
]0,1[\ne\emptyset.
\]
We note that if $\gamma$ is an integer, then $\Vset$ consists of all positive
integers $v$ with $v\ge\gamma$, and  
\[
J_v=\bigg]\frac{\gamma}{v+1},\frac{\gamma}{v}\bigg[,\qquad v\in\Vset.
\]
More generally, if $v_0=v_0(\gamma):=\gamma-\{\gamma\}_1\in\Z_+$, then
\[
J_v=\bigg]\frac{\gamma}{v+1},\frac{\gamma}{v}\bigg[,\qquad v\in\Vset
\setminus\{v_0\},
\]
and
\[
\sigmat_\gamma(J_v)=]0,1[,\qquad u\in\Vset\setminus\{v_0\}.
\]
So, the deviation from the ``filling'' requirement is rather slight (only
one branch fails). 
Next, we quickly calculate the derivative of $\sigmat_\gamma$:
\begin{equation*}
|\sigmat'_\beta(x)|=\frac{\gamma}{x^2}\ge\beta>1,\qquad x\in\cup
\{J_v:\,v\in\Uset\},
\end{equation*}
so the uniform expansiveness condition is met already by $\sigmat_\beta$ (with
$m=1$). Moreover, 
\[
\frac{|\sigmat''_\gamma(x)|}{|\sigmat'_\gamma(x)|^2}\le \frac{2x}{\gamma}\le 2,
\qquad x\in\cup\{J_v:\,v\in\Vset\},
\]
so we also have the second derivative control.
Again, we cannot unfortunately rely on Theorem \ref{tf} to give us the 
uniqueness of the absolutely continuous invariant probability measure for 
$\sigmat_\gamma$, as it apparently allows for non-uniqueness (although 
Remark \ref{rem-5.2} says that in the ``filling'' case we have the needed 
uniqueness). We also cannot rely on  Adler's theorem, as $\sigmat_\gamma$ 
is not necessarily a Markov map. However, it appears that here, it is 
nevertheless known that the absolutely continuous $\sigmat_\gamma$-invariant 
probability measure is unique and has strictly positive density almost 
everywhere. One way to do this is to check that condition $(iii)$ of Adler's
theorem is fulfilled for all $n\ge n(u)$, and proceed in an analogous fashion
as we do for $\tau_\beta$, with $\beta>1$. We omit the details.
\end{subsection}

\begin{subsection}{The explicit calculation of invariant measures}

In general, the computation of the absolutely continuous invariant measures 
for $\tau_\beta$ (as well as for $\sigmat_\gamma$) is intractable. Only in a 
few particular cases is it possible to supply explicit expressions for the 
corresponding densities. They all correspond to values of the parameters 
for which we are dealing with Markov maps.
For instance, when $\beta>1$ is an odd integer, $\tau_\beta$ is ``filling'' 
[i.e., we have complete branches], and the unique $\tau_\beta$-invariant 
probability measure on $[-1,1]$ is given by
\[
\frac{c(\beta)}{1-(x/\beta)^2}1_{[-1,1]}(x)\diff x,\quad\text{where}\quad 
\frac{1}{c(\beta)}=\int_{-1}^1
\frac{\diff x}{1-(x/\beta)^2}=\frac{\beta}{2}\log\frac{\beta+1}{\beta-1}
\qquad(\beta=3,5,7,\ldots).
\]
It is more interesting that it is possible to obtain the 
$\tau_\beta$-invariant probability density in a more complicated situation,
when $\beta=\frac32$. The uniqueness and ergodicity of that measure will be 
obtained in Section \ref{euimf}. As for notation, we write $1_E$ for the 
characteristic function of a set $E\subset\R$, which equals $1$ on $E$ and 
vanishes elsewhere. 

\begin{prop}
$(\beta=3/2)$ The density of the unique ergodic $\tau_{3/2}$-invariant 
absolutely continuous probability measure is given by
\[
\varrho_0(x)=c_0\bigg\{\frac{1}{1-(2x/3)^2}\,1_{[-\frac12,\frac12]}(x)
+\frac{3/4}{(1-|x|/3)(1+2|x|/3)}
1_{[-1,1]\backslash[-\frac12,\frac12]}(x)\bigg\},
\]
where $c_0^{-1}=\frac32\log(5/2)$.
\end{prop}

\begin{proof}
To simplify the notation, we write $\varrho_1(x):=c_0^{-1}\varrho_0(x)$, 
so that $\varrho_1(x)$ stands for the bracketed expression. We note that
both $\varrho_0,\varrho_1$ are even functions. We need to check that
\begin{equation}
\sum_{j\in\Z^\times}\frac{3/2}{(2j-x)^2}\varrho_1\bigg(\frac{3/2}{2j-x}\bigg)
=\varrho_1(x),\qquad x\in[-1,1],
\label{eq-tocheck}
\end{equation}
where this equality should be understood in the almost-everywhere sense.
Since
\[
\frac{3/2}{2j-x}\in[-\tfrac12,\tfrac12]\quad\text{for}\quad |j|\ge2,\,\,\,
x\in[-1,1],
\]
we may evaluate the sum of all but two terms in the left-hand side of 
\eqref{eq-tocheck}, as most of the terms cancel:
\begin{multline*}
\sum_{j:|j|\ge2}\frac{3/2}{(2j-x)^2}\varrho_1\bigg(\frac{3/2}{2j-x}\bigg)
=\sum_{j:|j|\ge2}\frac{3/2}{(2j-x)^2}\times\frac{1}{1-\frac{1}{(2j-x)^2}}
\\
=\sum_{j:|j|\ge2}\frac{3/2}{(2j-x)^2-1}=\frac{3}{4}
\sum_{j:|j|\ge2}\bigg[\frac{1}{2j-x-1}-\frac{1}{2j-x+1}\bigg]
\\
=\frac{3}{4}\bigg[\frac{1}{3-x}+\frac{1}{3+x}\bigg]=\frac{1/2}{1-(x/3)^2},
\qquad x\in[-1,1].
\end{multline*}
Next, we see that 
\[
\frac{3/2}{2-x}\in ]\tfrac12,1],\qquad x\in]\!-\!1,\tfrac12],
\]
and that this expression is in $]1,\frac32]$ for $x\in]\frac12,1]$. Here, we 
may of course replace $x$ by $-x$ if we make the necessary adjustments.  
It follows that
\begin{multline*}
\sum_{j:|j|=1}\frac{3/2}{(2j-x)^2}\varrho_1\bigg(\frac{3/2}{2j-x}\bigg)
=\frac{3/2}{(2-x)^2}\varrho_1\bigg(\frac{3/2}{2-x}\bigg)
+\frac{3/2}{(2+x)^2}\varrho_1\bigg(-\frac{3/2}{2+x}\bigg)
\\
=\frac{3/2}{(2-x)^2}\times\frac{3/4}{\big(1-\frac{1/2}{2-x}\big)
\big(1+\frac{1}{2-x}\big)}1_{[-1,\frac12]}(x)
+\frac{3/2}{(2+x)^2}\times\frac{3/4}{\big(1-\frac{1/2}{2+x}\big)
\big(1+\frac{1}{2+x}\big)}1_{[-\frac12,1]}(x)
\\
=\frac{1/4}{\big(1-\frac{2x}{3}\big)
(1-\frac{x}{3})}1_{[-1,\frac12]}(x)
+\frac{1/4}{\big(1+\frac{2x}{3}\big)\big(1+\frac{x}{3}\big)}
1_{[-\frac12,1]}(x)
\qquad x\in]\!-\!1,1[,
\end{multline*}
We may now express the whole sum in \eqref{eq-tocheck}:
\begin{multline*}
\sum_{j\in\Z^\times}\frac{3/2}{(2j-x)^2}\varrho_1\bigg(\frac{3/2}{2j-x}\bigg)
=\frac{1/2}{1-(\frac{x}{3})^2}+\frac{1/4}{\big(1-\frac{2x}{3}\big)
(1-\frac{x}{3})}1_{[-1,\frac12]}(x)
+\frac{1/4}{\big(1+\frac{2x}{3}\big)\big(1+\frac{x}{3}\big)}
1_{[-\frac12,1]}(x)\\
=
\frac{1}{4}\bigg[\frac{1}{1-\frac{x}{3}}+\frac{1}{1+\frac{x}{3}}\bigg]
+\frac{1}{4}\bigg[\frac{2}{1-\frac{2x}{3}}-\frac{1}{1-\frac{x}{3}}\bigg]
1_{[-1,\frac12]}(x)
+\frac{1}{4}\bigg[\frac{2}{1+\frac{2x}{3}}-\frac{1}{1+\frac{x}{3}}\bigg]
1_{[-\frac12,1]}(x)
\\
=\frac14\bigg[\frac{2}{1-\frac{2x}{3}}+\frac{2}{1+\frac{2x}{3}}\bigg]
1_{[-\frac12,\frac12]}(x)+\frac14\bigg[
\frac{1}{1-\frac{x}{3}}+\frac{2}{1+\frac{2x}{3}}\bigg]1_{]\frac12,1]}(x)
+\frac14\bigg[
\frac{1}{1+\frac{x}{3}}+\frac{2}{1-\frac{2x}{3}}\bigg]1_{]\!-\!1,-\frac12[}(x)
\\
=\frac{1/2}{1-(\frac{2x}{3})^2}1_{[-\frac12,\frac12]}(x)+
\frac{3/4}{(1-\frac{|x|}{3})(1+\frac{2|x|}{3})}
1_{[-1,1]\backslash[-\frac12,\frac12]}(x)=\varrho_1(x),
\qquad x\in]\!-\!1,1[.
\end{multline*}
The constant $c_0$ is determined by the requirement that we should have a 
probability density and easily computed. The proof is complete.
\end{proof}

\begin{rem}
$(a)$ It is possible to establish with similar means
the $\tau_\beta$-invariant absolutely continuous probability measure
for $\beta=n(2n+1)/(n+1)$, where $n$ is a positive integer.

\noindent{$(b)$} 
We mention here that the analogous $\sigmat_\gamma$-invariant absolutely 
continuous probability measures are known explicitly for $\gamma\in\Z_+$, 
see, e.g., \cite{ChRa}:
\[
\frac{c(\gamma)}{1+x/\gamma}1_{[0,1]}(x)\diff x,\quad\text{where}\quad 
\frac{1}{c(\gamma)}=\int_{0}^1
\frac{\diff x}{1+x/\gamma}=\gamma\log(1+1/\gamma)
\qquad(\gamma=1,2,3,\ldots).
\]
\end{rem}

\end{subsection}
\end{section}

\begin{section}{Characterization of the pre-annihilator space 
$\mathcal{M}_\beta^\bot$}
\label{notation}

\begin{subsection}{Purpose of the section; some notation}
In this section we provide a characterization of the subspace 
$\mathcal{M}_\beta^\bot$ in terms of certain operators.
We proceed in a fashion somewhat similar to that the one used in the proof 
of Lemma 5.2 in \cite{hh}. We recall from Subsection \ref{subsec-PFBI} 
that the Koopman operator for $\tau_\beta$ is denoted by $\Cop_{\tau_\beta}$, 
and recall from Subsection \ref{subsec-GTM1} that the 
corresponding Perron-Frobenius operator is $\Pop_{\tau_\beta}:L^1([-1,1])\to 
L^1([-1,1])$, given by
\[
\Pop_{\tau_\beta} h(x)=\sum_{j\in\Z^\times}
\frac{\beta}{(2j-x)^2} \,h\left(\frac{\beta}{2j-x}\right),\qquad x\in[-1,1],
\] 
with the understanding that $h\in L^1([-1,1])$ vanishes off the interval 
$[-1,1]$. Following the notation of \cite{hh}, we denote by 
$L_2^\infty(\R)$ the subspace of $L^\infty(\R)$ of $2$-periodic functions,
which is the same as the weak-star closure of 
\[
\mbox{span}\{\e^{in\pi x}:n\in\Z\}.
\]
Likewise, we let $L^\infty_{\langle\beta\rangle}(\R)$ be the weak-star closure 
of
\[
\mbox{span}\{\e^{i n\pi\beta/x}:n\in\Z\},
\]
which also has a characterization in terms of periodicity 
[$f\in L^\infty_{\langle\beta\rangle}(\R)$ if and only if the function 
$f(\beta/x)$ is in $L^\infty_2(\R)$].
Let $\Sop:L^\infty([-1,1])\longrightarrow L^\infty
\left(\R\backslash[-1,1]\right)$ be the operator defined by
\begin{equation}
\Sop g(x)=g(\{x\}_2),\qquad x\in\R\setminus[-1,1],
\label{eq-Sop1}
\end{equation}
and let $\Tope:L^\infty\left(\R\backslash[-\beta,\beta]\right)
\longrightarrow L^\infty([-\beta,\beta])$ be the operator given by
\begin{equation}
\Tope g(x)=g\left(-\frac{\beta}{\left\{-\beta/x\right\}_2}\right),\qquad
x\in [-\beta,\beta]\setminus\{0\}.
\label{eq-Tope1}
\end{equation}
It is clear that $\Sop$ and $\Tope$ are linear operators and that
they both have norm $1$ on the $L^\infty$ spaces where each one is defined. 
As a consequence, their pre-adjoints $\Sop^\ast$ and $\Tope^\ast$,
are norm contractions on the corresponding $L^1$ spaces.
The way things are set up, we have
\begin{eqnarray}
L^\infty_2(\R)&=&\{g+\Sop g:\,g\in L^\infty(-1,1)\},\label{1}\\
L^\infty_{\langle\beta\rangle}(\R)&=&\{g+\Tope g:\,g\in 
L^\infty(\R\backslash[-\beta,\beta])\}.\label{2}
\end{eqnarray}
We need the following restriction operators (recall that $\beta >1$),
\[
\begin{array}{llll}
\Rop_1&:&L^\infty(\R\backslash[-1,1])\longrightarrow 
L^\infty(\R\backslash[-\beta,\beta]), \\
\Rop_2&:&L^\infty([-\beta,\beta])\longrightarrow L^\infty([-1,1]),\\
\Rop_3&:&L^\infty([-\beta,\beta])\longrightarrow 
L^\infty([-\beta,\beta]\backslash[-1,1]),\\
\Rop_4&:&L^\infty(\R\backslash[-1,1])\longrightarrow 
L^\infty([-\beta,\beta]\backslash[-1,1]).
\end{array}
\]
These operators just restrict the given function to a subset, which 
make each one a norm contraction. The corresponding pre-adjoints 
$\Rop_i^\ast$, for $i=1,2,3,4$, act on the corresponding $L^1$ spaces, and
just extend the given function to a larger set by setting it equal to zero
where it was previously undefined.
As $\beta>1$, we easily check that $\Cop_\beta^2=\Rop_2\Tope
\Rop_1\Sop$. Taking the pre-adjoint of both sides, we get
\begin{equation}
\Pop_{\tau_\beta}^2=\Sop^\ast \Rop_1^\ast\Tope^\ast \Rop_2^\ast.
\label{pby}
\end{equation}
\end{subsection}

\begin{subsection}{The characterization of the pre-annihiliator space 
$\mathcal{M}_\beta^\bot$}
We now supply the criterion which characterizes when a given $f\in L^1(\R)$ 
belongs to $\mathcal{M}_\beta^\bot$.

\begin{prop}
\label{necsuf} 
$(1<\beta<+\infty)$ 
Let $f\in L^1(\R)$ be written as 
\[
f=f_1+f_2+f_3,
\]
where $f_1\in L^1([-1,1])$, $f_2\in L^1([-\beta,\beta]\backslash[-1,1])$, and
$f_3\in L^1(\R\backslash[-\beta,\beta])$. Then $f\in\mathcal{M}_\beta^\bot$ 
if and only if
\begin{align*}
(\Iop-\Pop_{\tau_\beta}^2)f_1&=\Sop^\ast(-\Rop_4^\ast+ 
\Rop_1^\ast \Tope^\ast \Rop_3^\ast) f_2,
\tag{$i$}
\\
f_3&=-\Tope^\ast \Rop_2^\ast f_1-\Tope^\ast \Rop_3^\ast f_2,
\tag{$ii$}
\end{align*}
where $\Iop$ is the identity on $L^1([-1,1])$.

\begin{proof}
In view of the definition \eqref{aniquilator} of $\mathcal{M}_\beta^\bot$, 
and the representations \eqref{1} and \eqref{2} of $L^\infty_2(\R)$ and
$L^\infty_{\langle\beta\rangle}(\R)$, we have that $f=f_1+f_2+f_3$ 
is in $\mathcal{M}_\beta^\bot$ if and only if
\begin{align*}
\langle f,g+\Sop g\rangle_\R=\langle f_1+f_2+f_3,g+\Sop g\rangle_\R
=&\,0,\qquad g\in L^\infty([-1,1]),\\
\langle f,h+\Tope h\rangle_\R=\langle f_1+f_2+f_3,h+\Tope h\rangle_\R 
=&\,0,\qquad h\in 
L^\infty(\R\backslash[-\beta,\beta]).
\end{align*}
Here, it is assumed that all functions $f_1,f_2,f_3$ are understood to
vanish outside their domain of definition. We see that the above equations
simplify to 
\begin{align*}
\langle f_1,g\rangle_{[-1,1]}+\langle f_2+f_3,\Sop 
g\rangle_{\R\backslash[-1,1]}
=&\,0,\qquad g\in L^\infty([-1,1]),\\
\langle f_3,h\rangle_{\R\backslash[-\beta,\beta]}
+\langle f_1+f_2,\Tope h\rangle_{[-\beta,\beta]} 
=&\,0,\qquad h\in 
L^\infty(\R\backslash[-\beta,\beta]).
\end{align*}
These equations are equivalent to having
\begin{eqnarray*}
f_1&=-\Sop^\ast (f_2+f_3),\\
f_3&=-\Tope^\ast(f_1+f_2).
\end{eqnarray*}
A more precise formulation of this is:
\begin{eqnarray}
f_1&=&-\Sop^\ast \Rop_4^\ast f_2-\Sop^\ast \Rop_1^\ast f_3;
\label{akk}\\
f_3&=&-\Tope^\ast\Rop_2^\ast f_1-\Tope^\ast 
\Rop_3^\ast f_2
\label{utt}.
\end{eqnarray}
We note first that \eqref{utt} is the same as $(ii)$. 
Next, we substitute \eqref{utt} into \eqref{akk} and take into
account \eqref{pby}; the result is $(i)$. This completes the proof.
\end{proof}
\end{prop}

\end{subsection}
\end{section}

\begin{section}{Exterior spectrum of the Perron-Frobenius operator for 
the Gauss-type map on $[-1,1]$}
\label{euimf}

\begin{subsection}{Purpose of the section}

In this section we will show that $\lambda_1=1$ is a simple
eigenvalue of $\Pop_{\tau_\beta}$. 
This corresponds to having a unique absolutely continuous invariant 
probability measure for $\tau_\beta$ with $\beta>1$. We will also prove that
$\Spec_{\mathrm{p}}(\Pop_{\tau_\beta})\cap\partial \D=\{1\}$. 
In view of Thorem \ref{tf}, these properties correspond to $\tau_\beta$ 
possessing {\em strong mixing}, with exponential decay of correlations
(cf. \cite{PY}, p. 122; also, compare with weak mixing \cite{CFS}, p. 22,  
p. 29, and \cite{Ios}, p. 203). Another useful reference is \cite{Bal}. 
\end{subsection}

\begin{subsection}{The iterates of an interval}
We  need the following lemma.

\begin{lem}
\label{lem}
$(1<\beta+\infty)$
Let $J_0$ be an nonempty open interval contained in $[-1,1]$. Then, for large 
enough positive integers $n$, say $n\ge n_0$, we have 
$\tau^n_\beta(J_0)\supset]\!-\!1,1[$.

\begin{proof}
We begin with the initial observation that if $\tau_\beta^n(J_0)
\supset]\!-1\!,1[$ holds for $n=n_0$, then it also holds for all $n\ge n_0$,
as most of the branches are complete (at most two may be incomplete). 
\smallskip

\noindent{\bf Case I:} {\em $\beta$ is an odd integer}. 
Then $\tau_\beta$ is ``filling'', that is, all branches are complete;
cf. Subsection \ref{subsec-dyn}. This case is well-understood, but it helps 
our presentation to take a look at it again.
We recall from Subsection \ref{subsec-dyn}
that the fundamental intervals are given by \eqref{eq-Iu1} with $\mathcal{U}$
being the set of all nonzero integers $u$ with $|u|\ge\frac12(\beta+1)$. 
We note that by \eqref{eq-expand}, 
$\tau_\beta$ is expansive: as long as an interval $J$ is contained in one 
of the fundamental intervals $I_u$, $u\in\mathcal{U}$, the image 
$\tau_\beta(J)$ is an interval of length at least $\beta$ times the length 
of $J$. 
We observe that if our given interval $J_0$ contains one of the fundamental
intervals $I_u$, $u\in\mathcal{U}$, then we are finished, because 
$\tau_\beta(J_0)\subset]\!-\!1,1[$ in this case.  
There are two other possibilities:
\smallskip

\noindent
\textit{$(a)$ The interval $J_0$ is contained in $I_u$ for some 
$u\in\mathcal{U}$:} 
In this case $\tau_\beta(J_0)$ is an interval of length at least 
$\beta m(J_0)$, by \eqref{eq-expand}.

\noindent
\textit{$(b)$ The interval $J_0$ has nonempty intersection with two 
neighboring fundamental intervals $I_u,I_{u'}$, and $J_0\subset
\mathrm{clos}\,[I_u\cup I_{u'}]$:}
In this case the length of the intersection of $J_0$ with
one of the two fundamental intervals, say with $I_u$, is at least 
$\frac12m(J_0)$. So, we have
\begin{equation}\label{fcil}
m(\tau_\beta(J_0))\geq m(\tau_\beta(J_0\cap I_u))\geq \frac{\beta}{2}m(J_0).
\end{equation}
In particular, $\tau_\beta(J_0)$ contains an interval 
$\tau_\beta(J_0\cap I_u)$ of length at least $\frac12\beta\, m(J_0)$.
\smallskip

\noindent
We see that in both cases $(a)$--$(b)$, the image $\tau_\beta(J_0)$ contains an
interval $J_1$ of length at least $\frac12\beta\, m(J)$. We note that $\frac12
\beta\ge\frac32$ as $\beta>1$ was an odd integer. By running the same argument
starting with $J_1$ in place of $J_0$, we see that unless $J_1$ contains a
fundamental interval (so that we are finished), we obtain an interval $J_2$
contained in $\tau_\beta(J_1)\subset\tau_\beta^2(J_0)$ of length at least
$(\frac12\beta)^2m(J_0)$. Continuing like this, we find successively intervals 
$J_1,J_2,J_3,\ldots$ of length $\ge(\frac12\beta)^lm(J_0)$ with
$J_l\subset\tau_\beta(J_{l-1})\subset\tau_\beta^l(J_0)$, and we stop only
when the interval $J_l$ contains a fundamental interval. For a big enough 
$l$ we must stop, at least for the reason that the length of $J_l$ will 
eventually exceed twice the maximum length of a fundamental interval, and  
for that $l$ we see that $\tau_\beta^{l+1}(J_0)\supset]\!-\!1,1[$. 
\smallskip

\noindent{\bf Case II:} {\em $\beta$ is not an odd integer}. 
Then $-1<\{\beta\}_2<1$,
and with $u_0:=\frac12(\beta-\{\beta\}_2)\in\Z_+$, $\mathcal{U}$ 
consists of all integers $u$ with $|u|\ge u_0$. We see that  
$\beta=2u_0+\{\beta\}_2\in]2u_0-1,2u_0+1[$.  
The fundamental intervals $I_u$ are given by \eqref{eq-Iu1} for 
$u\in\mathcal{U}\setminus\{\pm u_0\}$, while (cf. \eqref{eq-Iu})
\[
I_{u_0}=\Big]\frac{\beta}{2u_0+1},1\Big[,\qquad
I_{-u_0}=\Big]-1,-\frac{\beta}{2u_0+1}\Big[
\]
On a fundamental interval $I_u$, the transformation $\tau_\beta$ is given by 
$x\mapsto 2u-\beta/x$. 
\smallskip

\noindent{\bf Case II-A:} {\em $J_0$ is an edge fundamental interval, 
i.e., $J_0=I_{u_0}$ or $J_0=I_{-u_0}$}. 
When $J_0=I_{-u_0}$, we have
\begin{equation}
\label{termendo}
\tau_\beta(J_0)=\tau_\beta(I_{-u_0})
=]\beta-2u_0,1[\supset I_{u_0+1}^2:=\Big]\beta-2u_0,\frac{\beta}{2u_0+1}\Big[.
\end{equation}
If $\beta-2u_0\leq \beta/(2u_0+3)$, we have
\[
\tau_\beta(I_{-u_0})=]\beta-2u_0,1[\supset
I_{u_0+1}^2\supset
\Big]\frac{\beta}{2u_0+3},\frac{\beta}{2u_0+1}\Big[=I_{u_0+1},
\]
so that 
\[
\tau_\beta^2(I_{-u_0})\supset\tau_\beta(I_{u_0+1})=]\!-\!1,1[.
\]
It remains to treat the case when $\beta/(2u_0+3)<\beta-2u_0$, so that
$I_{u_0+1}^2\subset I_{u_0+1}$ and in particular, $\beta-2u_0\in I_{u_0+1}$. 
We first claim that there exists a constant $\beta'$ with $1<\beta'\le\beta$, 
which only depends on $\beta$, such that
\begin{equation}
\label{gamowskyo}
m\Big(\tau_\beta\Big(\Big]y,
\frac{\beta}{2u_0+1}\Big[\Big)\Big)\geq\beta' m(]y,1[),
\qquad y\in I_{u_0+1}^1:=\Big[\frac{\beta}{2u_0+3},\beta-2u_0\Big].
\end{equation}
Here, it is clear that $I_{u_0+1}^1\subset I_{u_0+1}$, and since 
$\tau_\beta$ is given by $x\mapsto2u_0+2-\beta/x$ on $I_{u_0+1}$, we have that
\[
\tau_\beta\Big(\Big]y,
\frac{\beta}{2u_0+1}\Big[\Big)=]2u_0+2-\beta/y,1[,
\]
We realize that the estimate \eqref{gamowskyo} is equivalent to having
\begin{equation}
\label{gamowskyo2}
\beta'y+\frac{\beta}{y}\geq 2u_0+1+\beta',
\qquad y\in I_{u_0+1}^1.
\end{equation}
The function $f(y)=\beta'y+\beta/y$ is strictly decreasing in the interval 
$]0,1]$, so it suffices to check \eqref{gamowskyo2} at the right endpoint of 
$I_{u_0+1}^1$. It is a straightforward exercise to verify that 
\eqref{gamowskyo2} holds at $y=\beta-2u_0=\{\beta\}_2$ provided that $\beta'$
is chosen sufficiently close to $1$; it helps to observe that $\beta>2$ 
holds because of $\beta/(2u_0+3)<\beta-2u_0$. So, \eqref{gamowskyo2} is valid
for $\beta'>1$ close enough to 1. If we put 
$J_1:=\tau_\beta(I_{-u_0})=]\beta-2u_0,1[$ and 
$J_2:=\tau_\beta(I^2_{u_0+1})=]\tau_\beta(\beta-2u_0),1[$ in accordance with
\eqref{termendo}, then, we have in view of \eqref{termendo} and 
\eqref{gamowskyo} that
\[
J_2\subset\tau_\beta (J_1)\quad\mbox{   and   }
\quad m(J_2)\geq \beta'm(J_1).
\]
If $\tau_\beta(\beta-2u_0)\le \beta/(2u_0+3)$, then $J_2\supset I_{u_0+1}$ and
so $\tau_\beta(J_2)\supset]\!-\!1,1[$ and hence $\tau_\beta^3(I_{-u_0})
\supset]\!-\!1,1[$. If not, then we rerun the argument and get ever bigger
intervals whose right endpoint is $1$, and eventually the interval must 
contain $I_{u_0+1}$, and we are finished.
The case $J=I_{u_0}$ is analogous and therefore omitted.
\smallskip

\noindent{\bf Case II-B:} {\em $J_0$ is a general nonempty open subinterval
of $[-1,1]$}.
We put
\[
x_0=\frac{(2u_0+1)\beta}{2u_0(2u_0+1)+\beta}.
\]
 The point $x_0$ belongs to the fundamental interval $I_{u_0}$ and has
\begin{equation}
\label{dicotomia1}
\tau_\beta\Big(\Big]\frac{\beta}{2u_0+1},x_0\Big[\Big)=I_{-u_0}\quad
\mbox{and}\quad
\tau_\beta\Big(\Big]-x_0,-\frac{\beta}{2u_0+1}\Big[\Big)=I_{u_0}.
\end{equation}
A little calculation shows that 
\begin{equation}
\tau'_\beta(x)\geq\frac{\beta}{x_0^2}>2,
\qquad x\in[-x_0,x_0]\cap\cup\{I_u:\,u\in\mathcal{U}\}.
\label{eq-expand3}
\end{equation}
Next, put $\beta'':=\min(\beta,\beta/(2x_0^2))$, so that $\beta''>1$. 
We have a general nonempty subinterval $J_0$ of $[-1,1]$, and want to show
that $\tau_\beta^n(J_0)$ covers $]\!-\!1,1[$ for some big enough $n$. We do
this by showing that the length of $\tau_\beta^n(J_0)$ must otherwise
continue growing geometrically. We observe that if $J_0$ contains one of the 
fundamental intervals $I_u$, $u\in\mathcal{U}$, then 
$\tau_\beta^n(J_0)\supset]\!-\!1,1[$ with $n=1$ for $|u|>u_0$, and with a 
possibly big $n$ if $|u|=u_0$ by Case II-A above.  
If $J_0$ does not contain a fundamental interval, then we are left with
the following two possibilities:
\smallskip

\noindent{$(a)$} The interval $J_0$ is contained in a fundamental interval
$I_u$ for some $u\in\mathcal{U}$. Then $J_1:=\tau_\beta(J_0)$ is an interval 
of length at least $\beta'' m(J_0)$, by \eqref{eq-expand}.
\smallskip

\noindent\textit{$(b)$ The interval $J_0$ has nonempty intersection with two 
neighboring fundamental intervals $I_u,I_{u'}$, and $J_0\subset
\mathrm{clos}\,[I_u\cup I_{u'}]$.}
In this case we have two subcases.
\smallskip

\noindent
\textit{$(b1)$ The interval $J_0$ is contained in $[-x_0,x_0]$.} 
Then one of the two intervals, say $I_u$, meets $J_0$ in a subinterval of 
length at least $\frac12m(J_0)$, and we have that 
$J_1:=\tau_\beta(J_0\cap I_u)$ is an open interval contained in 
$\tau_\beta(J_0)$ of length (cf. \eqref{eq-expand3})
\[
m(J_1)=m(\tau_\beta(J_0\cap I_u))\geq 
\frac{\beta m(J_0)}{2x_0^2}\ge\beta'' m(J_0).
\]

\noindent
\textit{$(b2)$ The interval $J_0$ is not contained in $[-x_0,x_0]$.} 
Then we have 
\[
J_0\cap I_{u_0}\supset 
\Big[\frac{\beta}{2u_0+1},x_0\Big]\quad\mbox{or}\quad
J_0\cap I_{-u_0}\supset \Big[-x_0,-\frac{\beta}{2u_0+1}\Big],
\]
so by \eqref{dicotomia1}, we have $\tau_\beta(J)\supset I_{-u_0}$ or 
$\tau_\beta(J)\supset I_{u_0}$.
\smallskip

\noindent
If $(b2)$ happens, we are finished, because after one iteration of 
$\tau_\beta$ we cover one of the edge fundamental intervals. If $(a)$ or $(b1)$
take place, then the set $\tau_\beta(J_0)$ contains an interval $J_1$ of 
length at least $\beta''$ times the length of $J_0$. We may then consider $J_1$
in place of $J_0$, and we gained that $J_1$ is longer. Unless we stop, which
is because the set contains a fundamental interval, we get a sequence of sets
$J_0,J_1,J_2,\ldots$, and their lengths grow geometrically. 
This is possible only finitely many times, which means that we eventually
cover a fundamental interval. 
The proof is complete.
\end{proof}
\end{lem}

\end{subsection}

\begin{subsection}{Exterior spectrum of the Perron-Frobenius operator}

For a real-valued function $f$, we use the standard convention to write
$f^+=\max\{f,0\}$ and $f^-=\max\{-f,0\}$, so that $f=f^+-f^-$. 

\begin{thm}\label{lemff}
$(1<\beta<+\infty)$
Let $\Pop_{\tau_\beta}$ be the Perron-Frobenius operator associated
to $\tau_\beta$ acting on $L^1([-1,1])$. Then
$\lambda_1=1$ is a simple eigenvalue of $\Pop_{\tau_\beta}$ and is
the only one with modulus one. Moreover, the eigenfunctions for eigenvalue
$1$ are of the form $c\varrho_0$, where $c\in\C$ is a constant, and 
$\varrho_0\diff m$ is the unique ergodic $\tau_\beta$-invariant absolutely
continuous probability measure. Also, $\varrho_0>0$ holds almost everywhere.

\begin{proof}
To simplify the notation, we write $\Pop$ in place of $\Pop_{\tau_\beta}$.
The transformation $\tau_\beta:]\!-\!1,1]\to]\!-\!1,1]$ is a partially 
filling $C^2$-smooth piecewise monotonic transformation, which meets the 
conditions
$(i)$ [with $m=1$] and $(ii)$ of Theorem \ref{tf}, so that by Remark 
\ref{rem-5.1} $(d)$, the assertion of Theorem \ref{tf} is 
valid also for $\tau_\beta$.  
So, we know that that $\lambda_1=1$ is an eigenvalue of $\Pop$ 
and so there is an eigenfunction $\varrho_0\in \BV([-1,1])$ 
corresponding to it. From \eqref{retardo} together with the triangle 
inequality, we see that 
$|\Pop\varrho_0|\leq \Pop|\varrho_0|$ point-wise.
Since $\Pop\varrho_0=\varrho_0$, we then have
\[
\int_{[-1,1]}|\varrho_0|\,\diff m
=\int_{[-1,1]}|\Pop\varrho_0|\,\diff m\leq \int_{[-1,1]} 
\Pop|\varrho_0|\,\diff m=\int_{[-1,1]}|\varrho_0|\,\diff m.
\]
This means that we must have $|\Pop\varrho|=\Pop|\varrho_0|$ 
almost everywhere on $[-1,1]$, and so in particular, 
$\Pop|\varrho_0|=|\varrho_0|$. But then $|\varrho_0|$ is another 
eigenfunction for $\lambda_1=1$. We might as well replace $\varrho_0$ by 
$|\varrho_0|$, which amounts to assuming that $\varrho_0\ge0$, and after
multiplication by a suitable positive constant we can assume that
\[
\langle\varrho_0,1\rangle_{[-1,1]}=\int_{[-1,1]}\varrho_0\,\diff m=1.
\]
We consider the set
\[
A_+:=\{x\in[-1,1]:\,\varrho_0(x)>0\}.
\]
Using \eqref{retardo}, we see that $\tau_\beta(A_+)\dot=A_+$ (the dot over
the equality sign means that the sets are equal up to Lebesgue null sets).
As an element of $\BV([-1,1])$, the function $\varrho_0$ can be 
assumed right-continuous. Then $A_+$ will contain some non-trivial open 
interval $I_0$. By iteration, we get that $\tau_\beta^n(A_+)\dot=A_+$
for $n=1,2,3,\ldots$, so, in particular, $A_+$ contains $\tau_\beta^n(I_0)$
(up to null sets). From Lemma \ref{lem} we know that for a large enough $n$, 
$\tau^n(I_0)$ covers $]\!-1\!,1[$, and so $\varrho_0>0$ holds
almost everywhere on $[-1,1]$. 

Next, we show that the eigenspace for $\lambda_1=1$ is one-dimensional.
We argue by contradiction, and suppose that that there exists a non-trivial
$\eta_1\in L^1([-1,1])$ so that $\varrho_0,\eta_1$ are linearly
independent, such that $\Pop\eta_1=\eta_1$.  
From Theorem \ref{tf} we know that $\eta_1\in\BV([-1,1])$. 
We consider the function
\[
f:=\{\langle\eta_1,1\rangle_{[-1,1]}\}\,\varrho_0-\eta_1\in 
\BV([-1,1]),
\] 
which has $\langle f,1\rangle_{[-1,1]}=0$, and $\Pop f=f$.
By replacing $\eta_1$ by its real or imaginary part (this is possible
since $\Pop$ preserves real-valuedness), we may assume that 
$\eta_1$ is real-valued, so that $f$ is real-valued. 
We can also assume that $f\in\BV([-1,1])$ is right-continuous. 
We now write $f=f^+-f^-$, and observe that $f^+,f^-$ are also 
right-continuous functions in $\BV([-1,1])$. Unless one of the
functions $f^+,f^-$ vanishes almost everywhere, both $f^+,f^-$ must be
positive on some open intervals $I_+,I_-$, respectively. In view of
\eqref{retardo}, we then have that for $n=1,2,3,\ldots$, the functions 
$\Pop^{n} f^+,\Pop^{n} f^-$ are positive almost everywhere on 
the sets $\tau^n(I_+),\tau^n(I_-)$, respectively. 
This means that Lemma \ref{lem} entails that there exists a positive integer
$n_0$ such that $\Pop^{n_0} f^+,\Pop^{n_0} f^-$
are both positive almost everywhere on $[-1,1]$. 
Since $\Pop^nf=f$ for all $n=1,2,3,\ldots$, we must then have
\begin{equation}
\label{contradi}
\|f\|_{L^1}=\|\Pop^{n_0} f\|_{L^1}=\|\Pop^{n_0}
f^+-\Pop^{n_0}f^-\|_{L^1}<\|\Pop^{n_0}
f^++\Pop^{n_0} f^-\|_{L^1}=\|f\|_{L^1},
\end{equation}
where the $L^1$ norm is with respect to the interval $[-1,1]$.
This contradiction shows that at least one of $f^+,f^-$ must
be identically zero (almost everywhere). But as 
$\langle f,1\rangle_{[-1,1]}=0$, both $f^+,f^-$ must then be the $0$ function.
This does not agree with our assumption that $\varrho_0,\eta_1$ should 
be linearly independent. We conclude that that $\lambda_1=1$ is a simple 
eigenvalue (i.e., it has a one-dimensional eigenspace). 

Next, we turn to the assertion that $\varrho_0\diff m$ is the unique ergodic 
absolutely continuous probability measure. This is a consequence of 
Proposition A.2.5 in \cite{Ios}, since we know that $\varrho_0\diff m$ is
an absolutely continuous measure with $\varrho_0>0$ almost everywhere, and
$\Pop\varrho_0=\varrho_0$.

Finally, we show that $\Pop$ has no other eigenvalues than $1$ on the 
unit circle $\T$. We let $\lambda$, with $|\lambda|=1$ and $\lambda\ne1$, be 
an eigenvalue of $\Pop$; $\eta_2\in L^1([-1,1])$ is the non-trivial 
eigenfunction corresponding to $\lambda$, which we may normalize: 
$\|\eta_2\|_{L^1}=1$.
From \eqref{retardo} together with the triangle inequality, we have 
$|\Pop\eta_2|\le\Pop|\eta_2|$ point-wise, and so, since 
$\Pop\eta_2=\lambda\eta_2$,
\[
\int_{[-1,1]}|\eta_2|\,\diff m
=\int_{[-1,1]}|\lambda\,\eta_2|\,\diff m
=\int_{[-1,1]}|\Pop\eta_2|\,\diff m\leq \int_{[-1,1]} 
\Pop|\eta_2|\,\diff m=\int_{[-1,1]}|\eta_2|\,\diff m.
\]
This means that we must have $|\Pop\eta_2|=\Pop|\eta_2|$ almost
everywhere on $[-1,1]$, and, consequently, $\Pop|\eta_2|=|\eta_2|$. 
But then $|\eta_2|=\varrho_0$, as the eigenspace for $\lambda_1=1$ was
one-dimensional and spanned by $\varrho_0$. We write $\eta_2=\chi\varrho_0$,
where the function $\chi\in L^\infty([-1,1])$ has $|\chi|=1$ 
almost everywhere. When we take another look at the argument we just did
involving equality in the triangle inequality, we realize that $\chi$ must 
have the property
\[
\chi(\tau_\beta(x))=\bar\lambda\,\chi(x),\qquad x\in[-1,1],
\]
in the almost-everywhere sense. By iteration, we get that
\begin{equation}
\label{eq-chi1}
\chi(\tau_\beta^n(x))=\bar\lambda^n\,\chi(x),\qquad x\in[-1,1],\,\,\,
n=1,2,3,\ldots.
\end{equation}
We pick a point $x_0$ where $\rho_0(x_0)>0$, and by right-continuity there
exists a non-empty (short) interval $]x_0,x_1[$ where $\rho_0,\eta_2$
are both very close to the value at $x_0$, so that $\chi$ is close to its
value at $x_0$ as well: say, for some small $\epsilon>0$,
\[
|\chi(x)-\chi(x_0)|<\epsilon,\qquad x\in]x_0,x_1[.
\]
Next, let $n$ be such that $\tau_\beta^n(]x_0,x_1[)\supset]-1,1[$. Then,
by \eqref{eq-chi1},
\[
|\chi(\tau_\beta^n(x))-\lambda^n\chi(x_0)|<\epsilon,\qquad x\in]x_0,x_1[,
\]
so that
\[
|\chi(y)-\lambda^n\chi(x_0)|<\epsilon,\qquad y\in]\!-\!1,1[.
\]
This means that $\chi$ is within a distance $\epsilon$ from a constant 
function.
As $\epsilon$ can be made as small as we like, the only possibility is that
$\chi$ is equal to a constant. But then \eqref{eq-chi1} is impossible unless 
$\lambda=1$, contrary to assumption.
The proof is complete.
\end{proof}
\end{thm}

\begin{rem} $(1<\beta<+\infty)$
A dynamical system $(I,\sigmaalg,\mu,\tau)$, where $\mu$ is finite and 
invariant under $\tau$, is said to be {\em exact} when we have 
\[
\lim_{n\rightarrow+\infty}\mu(\tau^n(A))=\mu(I),
\quad A\in\sigmaalg.
\]
It is known that when $(I,\sigmaalg,\mu,\tau)$ is exact, the $\tau$ possesses
strong mixing \cite{PY}, p. 125. 
If $\varrho_0$ stands for the density of the unique absolutely continuous 
$\tau_\beta$-invariant probability measure, then 
$([-1,1],\borelalg_{[-1,1]},\varrho_0\diff m,\tau_\beta)$ is exact. In 
particular,
\[
\lim_{n\rightarrow+\infty}m(\tau_\beta^n (A))=m([-1,1]),
\quad A\in\borelalg_{[-1,1]}.
\]
This can be obtained from Lemma \ref{lem} directly, by localizing around a
point of $A$ with density $1$ and using the distortion control available 
from the control on the second derivative; cf. \cite{hof}. 
This suggests another (shorter) way to obtain the assertion of Theorem
\ref{lemff}. We get that $\tau_\beta$ possesses strong mixing from the 
exactness, and then there can only be the eigenvalue $1$ and it must be 
simple.  
\end{rem}

\end{subsection}
\end{section}

\begin{section}{Proofs of the main results
}\label{pruebas}

\begin{subsection}{Proof of Theorem \ref{tpo}}\label{desc}
To prove Theorem \ref{tpo}, we will need to consider the space
$\BV([-\beta,\beta]\backslash[-1,1])$  of complex-valued integrable
functions defined on $[-\beta,-1]\cup[1,\beta]$ whose restrictions
to $[-\beta,-1]$ belong to $\BV([-\beta,-1])$ and whose restrictions to
$[1,\beta]$ belong to $\BV([1,\beta])$.

\begin{lem}
\label{ler}
$(1<\beta<+\infty)$ 
The operator $-\Sop^\ast \Rop_4^\ast+\Sop^\ast \Rop_1^\ast \Tope^\ast 
\Rop_3^\ast$ 
maps $\BV([-\beta,\beta]\backslash[-1,1])$ into 
$\BV([-1,1])$.

\begin{proof}
Let $f_2\in \BV([-\beta,\beta]\backslash[-1,1])$. 
The operator $\Sop^\ast:\,L^1(\R\backslash[-1,1])\to L^1([-1,1])$ is given by
\[
\Sop^\ast h(x)=\sum_{k\in\Z^\times}h(x+2k),\qquad x\in[-1,1].
\]
As 
\[
\Rop_4^\ast:\,L^1([-\beta,\beta]\backslash[-1,1])\to L^1(\R\backslash[-1,1])
\] 
extends the function to vanish where it was previously undefined, the function
$\Sop^\ast \Rop_4^\ast f_2$ is just a finite sum of functions of bounded 
variation, so that 
Then we have
\begin{equation}\label{despues}
\Sop^\ast \Rop_4^\ast f_2\in \BV([-1,1]).
\end{equation}
On the other hand, we can easily check that 
$\tilde{\Cop}^2=\Tope \Rop_1\Sop\Rop_2$, where 
$\tilde{\Cop}=\Cop_{\tilde \tau_\beta}$ is the Koopman operator 
associated to the transformation 
$\tilde\tau_\beta:[-\beta,\beta]\to[-\beta,\beta]$ given by
$\tilde\tau_\beta(0)=0$ and 
\[
\tilde{\tau}_\beta(x)=\{-\beta/x\}_2,\qquad x\in[-\beta,\beta]\setminus\{0\}.
\]
If $\tilde\Pop=\Pop_{\tilde\tau_\beta}:\,L^1([-\beta,\beta])\to
L^1([-\beta,\beta])$ is the corresponding Perron-Frobenius operator, whose
adjoint is $\tilde\Cop$, we find that
\begin{equation}
\label{k2}
\tilde{\Pop}^2=\Rop_2^\ast \Sop^\ast \Rop_1^\ast
\Tope^\ast. 
\end{equation} 
We easily check that $\tilde{\tau}_\beta$ satisfies conditions $(i)$
(with $m=2$) and $(ii)$ of Theorem. However, the mapping is only ``partially
filling''. But in view of Remark \ref{rem-5.1}, Theorem \ref{tf} remains
valid nevertheless. In particular, we have that 
$\tilde{\Pop}$ transforms $\BV([-\beta,\beta])$ into
itself. Since $\Rop_3^\ast$ and $\Rop_2^\ast$ are just extension by zero
operators, it follows from \eqref{k2} that
\begin{equation}
\label{ahora}
\Sop^\ast \Rop_1^\ast \Tope^\ast \Rop_3^\ast f_2\in \BV([-1,1]).
\end{equation}
Adding up, we get from \eqref{despues} and \eqref{ahora} that 
\[-\Sop^\ast \Rop_4^\ast f_2+\Sop^\ast \Rop_1^\ast 
\Tope^\ast \Rop_3^\ast f_2\in \BV([-1,1])
\] 
for every
$f_2\in \BV([-\beta,\beta]\backslash[-1,1])$.
The proof is complete.
\end{proof}
\end{lem}

We have now developed the tools needed to obtain Theorem \ref{tpo}. 
Actually, we formulate a more precise result. 

\begin{thm} $(1<\beta<+\infty)$ 
There exists a bounded operator $\Eop:\BV([-\beta,\beta]\backslash[-1,1])
\to L^1(\R)$ with the following properties:

\noindent{$(i)$} $\Eop$ is an extension operator, in the sense that 
$\Eop f(x)=f(x)$ almost everywhere on $[-\beta,\beta]\backslash[-1,1]$ for all 
$f\in\BV([-\beta,\beta]\backslash[-1,1])$.

\noindent{$(ii)$} The range of $\Eop$ is infinite-dimensional, and contained 
in $\mathcal{M}_\beta^\bot$.
\label{thm-8.2}
\end{thm}

\begin{proof}
To simplify the notation, we write $\Pop$ in place of $\Pop_{\tau_\beta}$.
By Theorem \ref{tf} (valid by Remark \ref{rem-5.1} $(d)$) together with
Theorem \ref{lemff}, we have the following representation for the iterates of 
$\Pop$,
\begin{equation}
\label{exp}
\Pop^n h =\{\langle h,1\rangle_{[-1,1]}\}\,\varrho_0 + 
\Zop^n h,\qquad h\in L^1([-1,1]),\,\,\,n=1,2,3,\ldots,
\end{equation}
where we write $\Zop$ in place of $\Zop_{\tau_\beta}$. Here, 
$\varrho_0\ge0$ is the density for the absolutely continuous 
$\tau_\beta$-invariant probability measure on $[-1,1]$; it has 
$\varrho_0\in\BV([-1,1])$. 
The operator $\Zop$ acts on $\BV([-1,1])$ and its
spectral radius is $<1$. 
By applying \eqref{exp} to $h=\varrho_0$, it is evident that 
$\Zop\varrho_0=0$. Next, if we note that
\[
\langle\Pop^n h,1\rangle_{[-1,1]}=\langle h,1\rangle_{[-1,1]},\qquad
n=1,2,3,\ldots,
\]
which is one of the standard properties of Perron-Frobenius operators
(e.g., we can use that $1$ is invariant under the Koopman operator), 
we also get that
\begin{equation}
\label{zerotal}
\langle\Zop^n h,1\rangle_{[-1,1]}=0,\qquad h\in L^1([-1,1]),\,\,\,
n=1,2,3,\ldots.
\end{equation}
Let us take an arbitrary element $f\in 
\BV([-\beta,\beta]\backslash[-1,1])$. From Lemma \ref{ler} above, we 
know that
\[
-\Sop^\ast \Rop_4^\ast f+\Sop^\ast \Rop_1^\ast \Tope^\ast 
\Rop_3^\ast f\in \BV([-1,1]).
\]
Since $\Zop$ has spectral radius $<1$ on $\BV([-1,1])$, 
the Neumann series
\[
(\Iop-\Zop^2)^{-1}=\Iop+\Zop^2+\Zop^4+\ldots
\]
converges to a bounded operator on $\BV([-1,1])$, and we may put
\begin{equation}
\label{auxiliar}
\Eop_1 f:=(\Iop-\Zop^2)^{-1}\Sop^\ast
\big\{-\Rop_4^\ast f+\Rop_1^\ast \Tope^\ast \Rop_3^\ast f\big\}\in\BV([-1,1]).
\end{equation}
We observe that 
\[
\big\langle-\Sop^\ast \Rop_4^\ast f+\Sop^\ast \Rop_1^\ast 
\Tope^\ast \Rop_3^\ast f,1\big\rangle_{[-1,1]}
=\langle f,1\rangle_{[-\beta,\beta]\backslash[-1,1]}
-\langle f,1\rangle_{[-\beta,\beta]\backslash[-1,1]}=0,
\]
and so, by \eqref{zerotal} and \eqref{auxiliar},
\begin{equation}
\label{eq-meanzero}
\langle \Eop_1 f,1\rangle_{[-1,1]}=\langle 
(\Iop-\Zop^2)\Eop_1 f,1\rangle_{[-1,1]}
=0.
\end{equation}
Finally, we put
\begin{equation}
\label{auxiliar2}
\Eop_3 f:=-\Tope^\ast( \Rop_2^\ast \Eop_1 +\Rop_3^\ast) f
\in L^1(\R\backslash[-\beta,\beta]).
\end{equation}
We define the operator $\Eop$ to be the mapping 
\[
\Eop:\,\BV([-\beta,\beta]\backslash[-1,1])\to L^1(\R),\,\,\,f
\mapsto f+\Eop_1 f+\Eop_3 f,
\]
with the understanding that each of the functions $f,\Eop_1 f,\Eop_2 f$ is 
extended to $\R$ by putting it equal to $0$ where it was previously undefined.
Then $\Eop$ is clearly bounded and linear, and has the property $(i)$.
In view of \eqref{eq-meanzero} and \eqref{exp}, we have 
$\Pop^n \Eop_1 f=\Zop^n \Eop_1 f$ for $n=1,2,3,\ldots$. This 
means that in condition $(i)$ of Proposition \ref{necsuf}, we may replace 
$\Pop^2$ by $\Zop^2$, and we just obtain the condition Proposition 
\ref{necsuf} $(i)$
rather immediately from \eqref{auxiliar}. The condition Proposition 
\ref{necsuf} $(ii)$ is immediate from \eqref{auxiliar2}. By Proposition 
\ref{necsuf}, we have that 
\[
\mathrm{im}\,\Eop\subset\mathcal{M}_\beta^\bot,
\] 
which proves $(ii)$, since $\mathrm{im}\,\Eop$ -- the range of $\Eop$ --
must be infinite-dimensional (the restriction to 
$[-\beta,\beta]\backslash[-1,1]$ of the range is
infinite-dimensional, being the space of all functions of bounded variation).
This completes the proof of Theorem \ref{tpo}.
\end{proof}

\begin{proof}[Proof of Theorem \ref{tpo}]
This is immediate from Theorem \ref{thm-8.2} $(ii)$.
\end{proof}

The next theorem shows that the range of $\Eop$ we constructed in
in the proof of Theorem \ref{tpo} is a subspace of the weighted space
$L^2(\R,\omega)$, where $\omega(x):=1+x^2$, with finite norm 
\[
\|f\|^2_{L^2(\R,\omega)}:=
\int_\R |f|^2\omega\,\diff m<+\infty.
\]

\begin{prop}
\label{nosecuanto}
The range of $\Eop$ is contained in $L^2(\R,\omega)$.

\begin{proof}
Let $f\in\BV([-\beta,\beta]\backslash[-1,1])$. Following the proof
of Theorem \ref{thm-8.2}, we see that $\Eop_1 f\in \BV([-1,1])$,
so that the restriction of $\Eop f$ to the interval $[-\beta,\beta]$
has bounded variation. In particular, $\Eop f$ is bounded on $[-\beta,\beta]$,
so we just to estimate the weighted $L^2$-norm integral on 
$\R\backslash[-\beta,\beta]$. The restriction of $\Eop f$ to 
$\R\backslash[-\beta,\beta]$ equals $\Eop_3f$, given by \eqref{auxiliar}, 
which we understand as $\Eop_3f=-\Tope^\ast h$. The operator 
$\Tope^\ast:\,L^1([-\beta,\beta])\to L^1(\R\backslash[-\beta,\beta])$ 
is given explicitly by
\begin{equation}
\label{tbstrella}
\Tope^\ast h(x)=\sum_{j\in\Z^\times} 
\frac{\beta^2}{(\beta+2jx)^2}\,h\left(\frac{\beta x}{\beta+2jx}\right),
\end{equation}
with the understanding that $h$ is extended to vanish off $[-\beta,\beta]$.
We write
\[
h_j(x)=h\left(\frac{\beta x}{\beta+2jx}\right)\frac{\beta^2}{(\beta+2jx)^2},
\qquad x\in\R\backslash[-\beta,\beta].
\]
As $h$ is bounded, it is clear from \eqref{tbstrella} that $\Eop_3f$ is 
bounded on $\R\backslash[-\beta,\beta]$. Since $\Eop_3 f$ is also summable, 
we must have
\[
\int_{\R\backslash[-\beta,\beta]}|\Eop_3 f(x)|^2\,\diff x<+\infty.
\]
Hence, it is enough to show that
\begin{equation}\label{poiuy}
\int_{\R\backslash[-\beta,\beta]}|\Eop_3 f(x)|^2x^2\,\diff x=
\int_{\R\backslash[-\beta,\beta]}
\Big|\sum_{j\in\Z^\times}
h_j(x)\Big|^2x^2\,
\diff x<+\infty.
\end{equation}
A straightforward computation shows that
\[
\int_{\R\backslash[-\beta,\beta]}|h_j(x)|^2x^2\,\diff x=
\int_{\frac{\beta}{2j+1}}^{\frac{\beta}{2j-1}}|h(x)|^2x^2\,\diff x
\leq \|h\|_{L^\infty([-\beta,\beta])}^2
\int_{\frac{\beta}{2j+1}}^{\frac{\beta}{2j-1}}
x^2\diff x\leq 
\frac{\beta^2}{j^4}\,\|h\|_{L^\infty([-\beta,\beta])}^2.
\]
As a consequence, we obtain
\[
\sum_{j\in\Z^\times}
\bigg\{\int_{\R\backslash[-\beta,\beta]}|h_j(x)|^2x^2\,\diff x
\bigg\}^{\frac{1}{2}}<+\infty,
\]
which entails \eqref{poiuy}. The proof is complete.
\end{proof}
\end{prop}
\noindent

\begin{rem}
The range of the operator $\Eop$ is a proper subspace of 
$\mathcal{M}_\beta^\bot$, even if we consider the closure of the range. 
Actually, if in the context of Proposition \ref{necsuf} we plug in 
$f_1:=\varrho_0$ (notation as in Theorem \ref{lemff}) and $f_2:=0$, 
and put $f_3:=-\Tope^\ast\Rop_2\varrho_0$, we obtain a function 
$\psi_0:=f_1+f_2+f_3=\varrho_0-\Tope^\ast\Rop_2\varrho_0$ which is in
the annihilator $\mathcal{M}_\beta^\bot$, but $\psi_0$ is not in the closure
of the range of $\Eop$. So a natural question is whether
\[
\mathrm{span}\{\psi_0\}\oplus\mathrm{clos}[\mathrm{im}\Eop]
=\mathcal{M}_\beta^\bot
\]
holds. This would be quite reasonable from the point of view of the proof of 
Proposition \ref{necsuf}.
\end{rem}
\end{subsection}

\begin{subsection}{ Proof of Theorem \ref{tpo1001}}
We turn to the proof of Theorem \ref{tpo1001}. We know from \cite{HM2}
that $\mathcal{N}_\beta^\bot$ is one-dimensional for $\beta=2$. 
We shall write $\beta=2\gamma$, and suppose that $\gamma>1$. We need the 
operators
\[
\Sop_+:\,L^\infty([0,1])\to L^\infty([1,+\infty[),\quad
\Sop_+ g(x):=g(\{x\}_1),
\]
and
\[
\Tope_+:\,L^\infty([\gamma,+\infty[)\to L^\infty([0,\gamma]),\quad
\Tope_+ g(x):=g\bigg(\frac{\gamma}{\{\gamma/x\}_1}\bigg).
\]
Their pre-adjoints map contractively 
\[
\Sop_+^\ast:\,L^1([1,+\infty[)\to L^1([0,1]),\quad
\Tope_+^\ast:\,L^1([0,\gamma])\to L^1([\gamma,+\infty[). 
\]
We need the following restriction operators:
\[
\begin{array}{llll}
\Rop_5&:&L^\infty([1,+\infty[)\longrightarrow 
L^\infty([\gamma,+\infty[), \\
\Rop_6&:&L^\infty([0,\gamma])\longrightarrow L^\infty([0,1]),\\
\Rop_7&:&L^\infty([0,\gamma])\longrightarrow 
L^\infty([1,\gamma]),\\
\Rop_8&:&L^\infty([1,+\infty[)\longrightarrow 
L^\infty([1,\gamma]).
\end{array}
\]
and their pre-adjoints $\Rop_5^\ast,\Rop_6^\ast,\Rop_7^\ast,\Rop_8^\ast$,
which act on the corresponding $L^1$-spaces. We let 
$\Pop_+:=\Pop_{\sigmat_\gamma}$ denote the Perron-Frobenius operator associated
to the Gauss-type transformation $\sigmat_\gamma:[0,1[\to[0,1[$, where 
$\sigmat_\gamma(0):=0$ and $\sigmat_\gamma(x):=\{\gamma/x\}_1$ for $x\in]0,1[$;
cf. Section \ref{Branch}.  The analogue of \eqref{pby} in this context is
\begin{equation}
\Pop_{+}^2=\Sop_+^\ast \Rop_5^\ast\Tope_+^\ast \Rop_6^\ast.
\label{eq-pby2}
\end{equation}
We also have an analogue of Proposition \ref{necsuf}.

\begin{prop}
\label{necsuf2} 
$(1<\gamma<+\infty)$ 
Let $f\in L^1([0,+\infty[)$ be written as 
\[
f=f_1+f_2+f_3,
\]
where $f_1\in L^1([0,1[)$, $f_2\in L^1([1,\gamma[)$, and
$f_3\in L^1([\gamma,+\infty[)$. Then 
$f\in\mathcal{N}_{2\gamma}^\bot$ if and only if
\begin{align*}
(\Iop-\Pop_{+}^2)f_1&=\Sop_+^\ast(-\Rop_8^\ast+ 
\Rop_5^\ast\Tope_+^\ast \Rop_7^\ast) f_2,
\tag{$i$}
\\
f_3&=-\Tope_+^\ast\Rop_6^\ast f_1-\Tope_+^\ast \Rop_7^\ast f_2,
\tag{$ii$}
\end{align*}
where $\Iop$ is the identity on $L^1([0,1])$.
\end{prop}

The proof is completely analogous to that of Proposition \ref{necsuf}, and
we omit it. Since $\gamma>1$, the transformation $\sigmat_\gamma$ is uniformly 
expansive, and if we analyze its spectral properties on $\mathrm{BV}([0,1])$,
we see that $\Pop_+$ has a spectral gap. More precisely, in the
context of Theorem \ref{tf} [valid by Remark \ref{rem-5.1} $(d)$], 
we can show that $\Spec(\Pop_+)\cap\T=\{1\}$ and that the eigenvalue 
$1$ is simple. This leads to the following assertion, analogous to 
Theorem \ref{thm-8.2}. Again, we omit the proof. 

\begin{thm} $(1<\gamma<+\infty)$ 
There exists a bounded operator $\Eop_+:\BV([1,\gamma])
\to L^1([0,+\infty[)$ with the following properties:

\noindent{$(i)$} $\Eop_+$ is an extension operator, in the sense that 
$\Eop_+ f(x)=f(x)$ almost everywhere on $[1,\gamma]$ for all 
$f\in\BV([1,\gamma])$.

\noindent{$(ii)$} The range of $\Eop_+$ is infinite-dimensional, and contained 
in $\mathcal{N}_{2\gamma}^\bot$.
\label{thm-8.6}
\end{thm}

\begin{proof}[Proof of Theorem \ref{tpo1001}]
This is immediate from Theorem \ref{thm-8.6} $(ii)$.
\end{proof}
\end{subsection}
\end{section}

\begin{section}{Final remarks}\label{applic}

\begin{subsection}{A related problem in the Hardy space of the unit disk}
An algebra of inner functions in the Hardy spaces $H^p$ of the unit disk was 
considered by Matheson and Stessin in \cite{ms}. This algebra depends on a 
parameter $\beta>0$, and can be assumed to be
\begin{equation}
\label{algebra}
\mathcal{A}_\beta=\mbox{span}\left\{\e^{-\pi m\frac{1-z}{1+z}}\,
\e^{-\pi n\beta\frac{1+z}{1-z}}:\,m,n=0,1,2,\ldots\right\},
\end{equation}
where the span is in the sense of finite linear combinations.
The main result \cite{ms} asserts that for any finite $p$, 
$\mathcal{A}_\beta$ is dense in $H^p$ for $\beta<1$, while it fails to
be dense for $\beta>1$. It is natural to also consider the (smaller) space
\begin{equation}
\label{subspace}
\mathcal{S}_\beta=\mbox{span}\left\{\e^{-\pi m\frac{1-z}{1+z}},
\,\e^{-\pi n\beta\frac{1+z}{1-z}}:\,m,n=0,1,2,\ldots\right\}.
\end{equation}
For $n=0,1,2,\ldots,$, let $\varphi^n$ denote the function 
\[
\varphi^n(z):=\e^{-\pi n\frac{1-z}{1+z}}\,
\e^{-\pi n\beta\frac{1+z}{1-z}},\qquad z\in\D.
\]
Then we clearly have the decomposition
\[
\mathcal{A}_\beta=\mathcal{S}_\beta+\varphi\mathcal{S}_\beta+
\varphi^2\mathcal{S}_\beta+\ldots,
\]
in the sense of finite sums.
As a consequence of the main result in \cite{hh}, Theorem \ref{thh}, 
we know that $\mathcal{S}_1$ is dense in the weak-star topology of BMOA, and
hence in the norm topology in $H^p$ for all finite $p$. 
It would appear rather plausible that the codimension of the $H^p$-closure
of $\mathcal{A}_\beta$ might be finite $\beta>1$, as there are only finitely
many points in the unit disk which are not separated by the generating inner
functions. However, so far, we cannot provide an answer to this question. 
With the aid of Theorem \ref{tpo} we can however prove that the $H^2$-closure
of the space $\mathcal{S}_\beta$ has infinite codimension for $\beta>1$. 
An outline of the argument is provided below.

We assume $\beta>1$, and pick an arbitrary natural number $N\ge1$. 
By Theorem \ref{thm-8.2}, we can pick linearly independent elements 
$f_1,\ldots,f_N\in\mathcal{M}_\beta^\bot$ in the range of 
$\Eop$, which all vanish on some proper subinterval of $[1,\beta]$, say on 
$[1,\beta']$, where $1<\beta'<\beta$. We define linearly independent 
functions $\tilde f_j$, $j=1,\ldots,N$, on the unit circle $\T$ as follows:
\[
\frac{1}{1-\imag x}\,\tilde{f_j}\left(\frac{1+\imag x}{1-\imag x}\right)
=(1+ix)\,f_j(x),\qquad j=1,\ldots,N.
\]
Next, by Proposition \ref{nosecuanto}, we have that the functions 
$\tilde{f_j}$ belong to $L^2(\T)$ and that they all vanish on a certain  
arc of $\T$. Let $\Qop: L^2(\T)\to H^2$ denote the orthogonal (Szeg\"o)
projection. We claim that the projected functions $\Qop\tilde f_j$,
$j=1,\ldots,N$, are linearly independent. Indeed, a relation of the form
\[
\sum_{j=1}^N c_j\proj\tilde f_j=0,\qquad c_j\in\C,
\]
implies that the function $\tilde f_{\mathrm{sum}}:=
\sum_{j=1}^N c_j\tilde f_j$ belongs to 
$L^2(\T)\ominus H^2=\mathrm{conj}\,H^2_0$, where ``conj'' means complex 
conjugation, and $H^2_0$ is the subspace of $H^2$ of functions that vanish
at the origin. Now the function $\tilde f_{\mathrm{sum}}$ is in 
$\mathrm{conj}\,H^2_0$ and vanishes along an arc of the circle $\T$, so by
e.g. Privalov's theorem, $\tilde f_{\mathrm{sum}}=0$ on all of $\T$. From
this and the linear independence of the functions 
$\tilde f_1,\ldots, \tilde f_N$, we obtain $c_j=0$ for all $j=1,\ldots,N$.  
So, the projected functions $\Qop\tilde f_j$, $j=1,\ldots,N$, are linearly
independent, as claimed. Finally, we claim that $\Qop\tilde f_j$, 
$j=1,\ldots,N$, belong to the orthocomplement of $\mathcal{S}_\beta$ in
$H^2$. If $(\cdot,\cdot)_{H^2}$ denotes the sesquilinear inner product of
$H^2$, we calculate that for $m=0,1,2,\ldots$, 
\begin{multline*}
\big(\proj\tilde f_j,\e^{-\pi m\frac{1-z}{1+z}}\big)_{H^2}=
\frac{1}{2\pi}\int_{\T}\proj\tilde{f_j}(\zeta)
\,\mathrm{conj}\,\Big(\e^{-\pi m\frac{1-\zeta}{1+\zeta}}\Big)\,|\diff\zeta|
\\
=\frac{1}{2\pi}\int_{\T}\tilde{f_j}(\zeta)
\,\mathrm{conj}\,\Big(\e^{-\pi m\frac{1-\zeta}{1+\zeta}}\Big)\,|\diff\zeta|
=\frac{1}{\pi}\int_\R f_j(x)\e^{\imag \pi m x}\,\diff x
=0,
\end{multline*}
where the last equality uses that $f_j\in\mathcal{M}_\beta^\bot$. 
In the analogous fashion, we get that for $n=0,1,2,\ldots$,
\begin{multline*}
\big(\proj\tilde f_j,\e^{-\pi\beta n\frac{1+z}{1-z}}\big)_{H^2}=
\frac{1}{2\pi}\int_{\T}\proj\tilde{f_j}(\zeta)
\,\mathrm{conj}\,\Big(\e^{-\pi\beta n\frac{1+\zeta}{1-\zeta}}\Big)
\,|\diff\zeta|
\\
=\frac{1}{2\pi}\int_{\T}\tilde{f_j}(\zeta)
\,\mathrm{conj}\,\Big(\e^{-\pi\beta n\frac{1+\zeta}{1-\zeta}}\Big)
\,|\diff\zeta|
=\frac{1}{\pi}\int_\R f_j(x)\e^{-\imag \pi\beta n/ x}\,\diff x
=0,
\end{multline*}
where again we use that $f_j\in\mathcal{M}_\beta^\bot$. 
It follows that the codimension of the $H^2$-closure of $\mathcal{S}_\beta$ 
must be $\ge N$. As $N$ was arbitrary, this means that the $H^2$-closure of 
$\mathcal{S}_\beta$ must have infinite codimension in $H^2$. 
\end{subsection}
\end{section}

\end{document}